\documentclass[10pt]{amsart}
\usepackage{amsmath,amssymb,latexsym,esint,cite,mathrsfs}
\usepackage{verbatim,wasysym}
\usepackage[left=2.5cm,right=2.5cm,top=2.4cm,bottom=2.4cm]{geometry}
\usepackage{tikz,enumitem,graphicx, subfig, microtype, color}
\usepackage{epic,eepic}

\usepackage[colorlinks=true,urlcolor=blue, citecolor=red,linkcolor=blue,
linktocpage,pdfpagelabels, bookmarksnumbered,bookmarksopen]{hyperref}
\usepackage[hyperpageref]{backref}
\usepackage[english]{babel}

\numberwithin{equation}{section}

\newtheorem{thm}{Theorem}[section]
\newtheorem{lem}[thm]{Lemma}
\newtheorem{cor}[thm]{Corollary}
\newtheorem{Prop}[thm]{Proposition}

\newcommand{\N}{\mathbb{N}}

\newcommand{\R}{\mathbb{R}}

\renewcommand{\i}{{\rm i}}

\def\de {\delta}

\def\va {\varphi}

\begin{document}
\baselineskip=14pt

\title[Critical Choquard equation]{Existence of solutions for critical Choquard equations via the concentration compactness method}

\author[F.\ Gao]{Fashun Gao}
\author[E.\ Silva]{Edcarlos D. da Silva}
\author[M.\ Yang]{Minbo Yang$^*$}
\author[J.\ Zhou]{Jiazheng Zhou}

\address{Fashun Gao, \newline\indent Department of Mathematics, Zhejiang Normal University, \newline\indent
	Jinhua 321004, People's Republic of China}
\email{fsgao@zjnu.edu.cn}

\address{Edcarlos D. da Silva, \newline\indent IME ¨C Universidade Federal de Goi\'as, \newline\indent
 74001-970, Goiania, GO, Brazil}
\email{eddomingos@hotmail.com}

\address{Minbo Yang, \newline\indent Department of Mathematics, Zhejiang Normal University, \newline\indent
	Jinhua 321004, People's Republic of China}
\email{mbyang@zjnu.edu.cn}

\address{Jiazheng Zhou,\newline\indent Universidade de Bras\'{\i}lia, Departamento de Matem\'atica, \newline\indent
	70910-900, Bras\'iliaDF, Brazil}
\email{jiazzheng@gmail.com}


\thanks{$^*$Minbo Yang is the corresponding author who is partially supported by NSFC(11571317,11671364).}

\begin{abstract}
In this paper we consider the nonlinear Choquard equation
$$
-\Delta u+V(x)u
=\left(\int_{\mathbb{R}^N}\frac{G(y,u)}{|x-y|^{\mu}}dy\right)g(x,u)\hspace{4.14mm}\mbox{in}\hspace{1.14mm} \mathbb{R}^N,
$$
where $0<\mu<N$, $N\geq3$,  $g(x,u)$ is of critical growth due to the Hardy--Littlewood--Sobolev inequality and $G(x,u)=\displaystyle\int^u_0g(x,s)ds$. Firstly, by assuming that the potential $V(x)$ might be sign-changing, we study the existence of Mountain-Pass solution via a concentration-compactness principle for the Choquard equation. Secondly, under the conditions introduced by Benci and Cerami \cite{BC1}, we also study the existence of high energy solution by using a global compactness lemma for the nonlocal Choquard equation.

\end{abstract}

\maketitle

\begin{center}
	\begin{minipage}{8.5cm}
		\small
		\tableofcontents
	\end{minipage}
\end{center}
%

\section{Introduction and main results}
The nonlinear Choquard equation
\begin{equation}\label{Nonlocal.S1}
-\Delta u+ V(x)u=\big(|x|^{-\mu}\ast |u|^{q}\big)|u|^{q-2}u,\hspace{4.14mm} \mbox{in}\ \mathbb{R}^N
\end{equation}
arises in various fields of mathematical physics, such as the description of the quantum theory of a polaron at rest by S. Pekar in 1954 \cite{Ps} and the modeling of an electron trapped
in its own hole in 1976 in the work of P. Choquard, as a certain approximation to Hartree-Fock theory of one-component plasma \cite{L1} The equation \eqref{Nonlocal.S1} is also known as the Schr\"{o}dinger-Newton equation \cite{Pe}, since the convolution part might be treated as a coupling with a Newton equation.

Mathematically, Lieb \cite{L1} proved the existence and uniqueness, up to translations, of the ground
state for \eqref{Nonlocal.S1} with $\mu=1$, $q=2$ and $V$ is a positive constant and Lions \cite{Ls} showed the existence of a sequence of radially symmetric solutions by variational methods. In the last decades, a great deal of mathematical efforts has been devoted to the study of existence, multiplicity and properties of the solutions of the nonlinear Choquard equation \eqref{Nonlocal.S1}. In \cite{CCS1, MZ,  MS1}, the authors showed the regularity, positivity and radial symmetry of the ground states and
derived decay property at infinity as well.  Moroz and Van Schaftingen also considered in \cite{MS3} the
existence of ground states under the assumption of Berestycki-Lions type. If the periodic potential $V(x)$ changes sign and $0$ lies in the gap of the spectrum of $-\Delta +V$, then the energy functional associated to the problem is strongly indefinite indeed. For this case, the existence of solution for $p=2$ was considered in \cite{BJS}. Later Ackermann \cite{AC} proposed a new approach to prove the existence of infinitely many geometrically distinct weak solutions. If the nonlinear Choquard equation is equipped with deepening potential well of the form $\lambda a(x)+1$ where $a(x)$ is a nonnegative continuous function such that $\Omega =$ int $(a^{-1}(0))$ is a non-empty bounded open set with smooth boundary, in \cite{ANY} the authors studied the existence and multiplicity of multi-bump shaped solution. The existence and concentration behavior of solutions for the singularly perturbed subcritical Choquard equation(Semiclassical Problems) have been considered in \cite{AY1, AY2, ACTY, AGSY, CCS, MS4, WW}, Wei
and Winter \cite{ WW} constructed families of solutions
 by a Lyapunov-Schmidt type reduction.
 Cingolani et.al.\cite{CCS} showed that there exists a family of solutions
having multiple concentration regions which are located around the minimum points of the potential. Moroz and Van Schaftingen \cite{MS4} developed a nonlocal penalization technique and showed the existence of a family of solutions
concentrating around the local minimum of $V$. In \cite{AY1, AY2}, Alves and Yang proved the existence, multiplicity and concentration of solutions for the equation by penalization method and Lusternik-Schnirelmann theory.

To consider the nonlocal elliptic equation involving Riesz type potential, it is necessary to recall the well--known Hardy--Littlewood--Sobolev inequality.
\begin{Prop}\label{HLS}
 (Hardy--Littlewood--Sobolev inequality). (See \cite{LL}.) Let $t,r>1$ and $0<\mu<N$ with $1/t+\mu/N+1/r=2$, $f\in L^{t}(\mathbb{R}^N)$ and $h\in L^{r}(\mathbb{R}^N)$. There exists a sharp constant $C(t,N,\mu,r)$, independent of $f,h$, such that
\begin{equation}\label{HLS1}
\int_{\mathbb{R}^{N}}\int_{\mathbb{R}^{N}}\frac{f(x)h(y)}{|x-y|^{\mu}}dxdy\leq C(t,N,\mu,r) |f|_{t}|h|_{r},
\end{equation}
where $|\cdot|_{q}$ for the $L^{q}(\mathbb{R}^{N})$-norm for $q\in[1,\infty]$. If $t=r=2N/(2N-\mu)$, then
$$
 C(t,N,\mu,r)=C(N,\mu)=\pi^{\frac{\mu}{2}}\frac{\Gamma(\frac{N}{2}-\frac{\mu}{2})}{\Gamma(N-\frac{\mu}{2})}\left\{\frac{\Gamma(\frac{N}{2})}{\Gamma(N)}\right\}^{-1+\frac{\mu}{N}}.
$$
In this case there is equality in \eqref{HLS1} if and only if $f\equiv Ch$ and
$$
h(x)=A(\gamma^{2}+|x-a|^{2})^{-(2N-\mu)/2}
$$
for some $A\in \mathbb{C}$, $0\neq\gamma\in\mathbb{R}$ and $a\in \mathbb{R}^{N}$.
\end{Prop}

Let $H^{1}(\R^N)$ be the usual Sobolev spaces with norm
    \[
    \|u\|_{H^1}:=\left(\int_{\R^N}(|\nabla u|^2+ |u|^2)dx\right)^{1/2},
    \]
    $D^{1,2}(\R^N)$ be equipped with norm
    $$
\|u\|:=\Big(\int_{\R^N}|\nabla u|^{2}dx\Big)^{\frac{1}{2}}
$$
and $L^s(\R^N)$, $1 \leq s \leq \infty$,
denotes the Lebesgue space with norms
\begin{gather*}
| u |_s:=\Big(\int_{\R^N}|u|^sdx\Big)^{1/s}.
\end{gather*}
By the Hardy--Littlewood--Sobolev inequality, for every $u\in H^{1}(\mathbb{R}^{N})$, the integral
$$
\int_{\mathbb{R}^{N}}\int_{\mathbb{R}^{N}}\frac{|u(x)|^{q}|u(y)|^{q}}{|x-y|^{\mu}}dxdy
$$
is well defined if
$$
\frac{2N-\mu}{N}\leq q\leq\frac{2N-\mu}{N-2}.
$$
Due to this fact, it is quite natural to call $\frac{2N-\mu}{N}$ the lower critical exponent and $2_{\mu}^{\ast}=\frac{2N-\mu}{N-2}$ the upper critical exponent.
In \cite{MS2, DSZ}, the authors considered the nonlinear Choquard equation \eqref{Nonlocal.S1} in $\R^N$ with lower critical exponent $\frac{2N-\mu}{N}$ and obtained some existence and nonexistence results.
In order to study the
critical nonlocal equation with upper critical exponent $2_{\mu}^{\ast}$, let $S$ be the best Sobolev constant defined by:
\[
S|u|^2_{2^*}\leq \int_{\R^N}|\nabla u|^2dx \ \ \ \hbox{for all $u\in
D^{1,2}(\R^N)$},
\]we will use $S_{H,L}$ to denote the best constant defined by
\begin{equation}\label{S1}
S_{H,L}:=\displaystyle\inf\limits_{u\in D^{1,2}(\mathbb{R}^N)\backslash\{{0}\}}\ \ \frac{\displaystyle\int_{\mathbb{R}^N}|\nabla u|^{2}dx}{\left(\displaystyle\int_{\mathbb{R}^N}\int_{\mathbb{R}^N}
\frac{|u(x)|^{2_{\mu}^{\ast}}|u(y)|^{2_{\mu}^{\ast}}}{|x-y|^{\mu}}dxdy\right)^{\frac{N-2}{2N-\mu}}}.
\end{equation}
In \cite{GY} it was observed that
\begin{Prop}\label{ExFu} (See \cite{GY}.)
The constant $S_{H,L}$ defined in \eqref{S1} is achieved if and only if $$u=C\left(\frac{b}{b^{2}+|x-a|^{2}}\right)^{\frac{N-2}{2}} ,$$ where $C>0$ is a fixed constant, $a\in \mathbb{R}^{N}$ and $b\in(0,\infty)$ are parameters. What's more,
$$
S_{H,L}=\frac{S}{C(N,\mu)^{\frac{N-2}{2N-\mu}}},
$$
where $S$ is the best Sobolev constant and $C(N,\mu)$ is given in Proposition \ref{HLS}.
\end{Prop}
Denote  $\widetilde{U}_{\delta,z}(x):=\frac{[N(N-2)\delta]^{\frac{N-2}{4}}}{(\delta+|x-z|^{2})^{\frac{N-2}{2}}}$, $\delta>0$, $z\in\mathbb{R}^{N}$.
 We know that $\widetilde{U}_{\delta,z}$ is a minimizer for $S$ \cite{Wi} and
\begin{equation}\label{REL}
\aligned
U_{\delta,z}(x):=C(N,\mu)^{\frac{2-N}{2(N-\mu+2)}}S^{\frac{(N-\mu)(2-N)}{4(N-\mu+2)}}\widetilde{U}_{\delta,z}(x)
\endaligned
\end{equation}
is the unique  minimizer for $S_{H,L}$ that satisfies
\begin{equation}\label{CCE1}
-\Delta u
=\left(\int_{\mathbb{R}^N}\frac{|u(y)|^{2_{\mu}^{\ast}}}{|x-y|^{\mu}}dy\right)|u|^{2_{\mu}^{\ast}-2}u\hspace{4.14mm}\mbox{in}\hspace{1.14mm} \mathbb{R}^N
\end{equation}
and
$$
\int_{\mathbb{R}^N}|\nabla U_{\delta,z}|^{2}dx=\int_{\mathbb{R}^N}\int_{\mathbb{R}^N}\frac{|U_{\delta,z}(x)|^{2_{\mu}^{\ast}}
|U_{\delta,z}(y)|^{2_{\mu}^{\ast}}}{|x-y|^{\mu}}dxdy=S_{H,L}^{\frac{2N-\mu}{N-\mu+2}}.
$$
In \cite{GY, GYm} the authors
considered the Br\'{e}zis-Nirenberg type problem
\begin{equation}\label{CCE}
-\Delta u
=\left(\int_{\Omega}\frac{|u(y)|^{2_{\mu}^{\ast}}}{|x-y|^{\mu}}dy\right)|u|^{2_{\mu}^{\ast}-2}u+\lambda u\hspace{4.14mm}\mbox{in}\hspace{1.14mm} \Omega
\end{equation}
and established the existence, multiplicity and nonexistence of solutions for the nonlinear Choquard equation in bounded domain. It is observed in \cite{GSY} that equation \eqref{CCE} can regarded as a limit problem for a critical Choquard equation with deepening potential well, there the existence and asymptotic behavior of the solutions were investigated.
In \cite{AGSY}, by investigating the ground states of the critical Choquard equation with constant coefficients, the authors studied the semiclassical limit problem for the singularly perturbed Choquard equation in $\R^3$ and characterized the concentration behavior by variational methods. The upper critical case with general nonlinearity was studied in \cite{CZ}. The planar case was considered in \cite{ACTY}, there the authors established the existence of ground state for the limit problem with critical exponential growth which complemented those results for local case, and then they also studied the concentration around the global minimum set. Gao and Yang in \cite{ GY3} investigated the existence result for the strongly indefinite Choquard equation with upper critical exponent in the whole space.

In works \cite{AGSY, ACTY, GY3}, the method developed by Brezis and Nirenberg has been successfully adopt to study the Choquard equation with upper critical exponents. There the authors are able to prove the existence results by showing that the minmiax value was below some critical criteria where the $(PS)$ condition still holds. In the present paper we continue to study the Choquard equation with upper critical exponents, but with different types of potential functions. We will see that the arguments in \cite{AGSY, ACTY, GY3} does not apply  for these new situations any longer.

On one hand we are going yo study the critical Choquard equation with subcritical perturbation and potential functions might change sign
\begin{equation}\label{CCe1}
\displaystyle-\Delta u+V(x)u
=\left(\int_{\mathbb{R}^N}\frac{|u(y)|^{2_{\mu}^{\ast}}+|u(y)|^{p}}{|x-y|^{\mu}}dy\right)
\Big(|u|^{2_{\mu}^{\ast}-2}u+ \frac{p}{2_{\mu}^{\ast}}|u|^{p-2}u\Big)\hspace{4.14mm}\mbox{in}\hspace{1.14mm} \R^{N},
\end{equation}
where $N\geq3$, $0<\mu<N$,  $(2N-\mu)/N<p<(2N-\mu)/(N-2)$ and $2_{\mu}^{\ast}=(2N-\mu)/(N-2)$ is the upper critical exponent in the sense of the Hardy--Littlewood--Sobolev inequality.
To obtain the existence result we are going to prove that the lack of compactness was recovered by using the concentration compactness principle. Following \cite{Z}, we will assume that the functions $V(x)$ satisfies the following condition:

$(V)$ There exists $\tau_{0}> 0$ such that the set $\Omega_{\tau_{0}}=\{x\in\mathbb{R}^N:V(x)\leq\tau_{0}\}$ has the finite Lebesgue measure. Moreover, $V\in L_{loc}^{\infty}(\mathbb{R}^N)\cap L^{\frac{N}{2}}(\R^N)$ and there holds
$$
V_{0}:=|V_{-}(x)|_{L^{N/2}}<S,
$$
where $S$ is the best Sobolev constant and $V_{-}=\max\{-V(x),0\}$.

We can draw the following conclusion.
\begin{thm}\label{EX1}
Suppose that assumption $(V)$ holds, $N\geq3$, $0<\mu<N$ and $(2N-\mu)/N<p<(2N-\mu)/(N-2)$. Then \eqref{CCe1} admits a nontrivial solution.
\end{thm}

On the other hand, we are interested in the existence of high energy solution for the critical Choquard equation. In the famous paper \cite{BC1}, Benci and Cerami considered the following problem
\begin{equation}\label{local.S1}
-\Delta u+V(x)u=|u|^{2^{\ast}-2}u,\hspace{4.14mm} \mbox{in}\ \mathbb{R}^N,
\end{equation}
where the potential $V(x)$ satisfies $(V_1)$, $(V_2)$ below and
$(V_3')$
$$
|V(x)|_{L^{N/2}}<S(2^{\frac{2}{N}}-1).
$$
They developed some global compactness lemma and proved that the problem \eqref{local.S1} has at least one positive high energy solution. Here we are quite interested if the same result still holds for the nonlocal Choquard equation
\begin{equation}\label{CE2}
\left\{\begin{array}{l}
\displaystyle-\Delta u+ V(x)u
=\left(\int_{\mathbb{R}^N}\frac{|u(y)|^{2_{\mu}^{\ast}}}{|x-y|^{\mu}}dy\right)
|u|^{2_{\mu}^{\ast}-2}u\hspace{4.14mm}\mbox{in}\hspace{1.14mm} \mathbb{R}^N,\\
\displaystyle u\in D^{1,2}(\mathbb{R}^N),\hspace{10.6mm}
\end{array}
\right.
\end{equation}
here $0<\mu<N$, $N\geq3$, $2_{\mu}^{\ast}=(2N-\mu)/(N-2)$  and the potential $V$ satisfies the assumptions

$(V_1)$ $V\in \mathcal{C}(\mathbb{R}^N,\mathbb{R})$, $V\geq\nu>0$ in a neighborhood of 0.

$(V_2)$ $\exists p_{1}<\frac{N}{2},p_{2}>\frac{N}{2}$ and for $N=3$, $p_{2}<3$, such that
$$
V(x)\in L^{p},\ \ \ \forall p\in[p_{1},p_{2}].
$$

$(V_3)$
$$
|V(x)|_{L^{N/2}}<C(N,\mu)^{\frac{N-2}{2N-\mu}}S_{H,L}(2^{\frac{N+2-\mu}{2N-\mu}}-1),
$$
where $S_{H,L}$ is defined in \eqref{S1} and $C(N,\mu)$ is given in Proposition \ref{HLS}. Under these assumptions, we have
\begin{thm}\label{EXS}
Suppose that assumptions $(V_1)$, $(V_2)$ and $(V_3)$ hold, $0<\mu<\min\{4,N\}$ and $N\geq3$. Then equation \eqref{CE2} has at least one nontrivial solution $u$.
\end{thm}

An outline of this paper is as follow: In Section 2, we prove a version of Concentration-Compactness principle for the nonlocal type problem which complements the results in \cite{Ls1, BCS, BTW}. After that we can use the compactness lemma to prove that the $(PS)$ condition still holds below some criteria level and obtain the existence of solutions by Mountain-Pass Theorem. In Section 3, we prove a version of global compactness lemma for the nonlocal Choquard equation and then we show the existence of high energy solution for \eqref{CE2} following the linking arguments in \cite{BC1}.

\section{Mountain-Pass solution}
In this section we will study the existence of solutions for equation \eqref{CCe1} under assumption $(V)$. To prove the existence of solutions by variational methods, we introduce the Hilbert spaces
\[
E:=\left\{u\in H^1(\R^N): \, \int_{\R^N}V_{+}(x)u^2dx<\infty\right\}
\]
with inner products
\[
(u, v):=\int_{\R^N}\big(\nabla u\nabla v+V_{+}(x)uv\big)dx
\]
and the associated norms
\[\|u\|^2_V=(u,u).
\]
 Obviously, $E$ embeds continuously in $H^1(\R^N)$ (see \cite{DL}). Moreover,
 \begin{lem}([\cite{Z}, Lemma 2.3])\label{F0}
There exist $C_1,C_2 > 0$ depending only on the structural constants
such that
\begin{equation}\label{p0}
C_1\|u\|_{H^{1}}^{2}\leq C_2\|u\|_{V}^{2}\leq\int_{\mathbb{R}^N}(|\nabla u|^{2}+V(x)|u|^{2})dx\leq\|u\|_{V}^{2}, \ u\in E.
\end{equation}
\end{lem}
Denote
$$
\|u\|_{NL}:=\left(\int_{\mathbb{R}^N}\int_{\mathbb{R}^N}\frac{|u(x)|^{2_{\mu}^{\ast}}|u(y)|^{2_{\mu}^{\ast}}}
{|x-y|^{\mu}}dxdy\right)^{\frac{1}{2\cdot2_{\mu}^{\ast}}},
$$
the following splitting Lemma was proved in Lemma 2.2 of \cite{GY}.
\begin{lem} \label{BLN}Let $N\geq3$ and $0<\mu<N$. If $\{u_{n}\}$ is a bounded sequence in $L^{\frac{2N}{N-2}}(\mathbb{R}^N)$ such that $u_{n}\rightarrow u$ almost everywhere in $\mathbb{R}^N$ as $n\rightarrow\infty$, then the following hold,
$$
\|u_{n}\|_{NL}^{2\cdot2_{\mu}^{\ast}}
-\|u_{n}-u\|_{NL}^{2\cdot2_{\mu}^{\ast}}\rightarrow\|u\|_{NL}^{2\cdot2_{\mu}^{\ast}}
$$
as $n\rightarrow\infty$.
\end{lem}

To study the problem variationally, we introduce the energy functional associated to equation \eqref{CCe1} by
$$
J(u)=\frac{1}{2}\int_{\mathbb{R}^N}(|\nabla u|^{2}+V(x)|u|^{2})dx-\frac{1}{2\cdot2_{\mu}^{\ast}}\int_{\mathbb{R}^N}
\int_{\mathbb{R}^N}\frac{(|u(x)|^{2_{\mu}^{\ast}}+|u(x)|^{p})(|u(y)|^{2_{\mu}^{\ast}}+|u(y)|^{p})}
{|x-y|^{\mu}}dxdy.
$$
The Hardy--Littlewood--Sobolev inequality implies that $J$ is well defined on $E$ and belongs to $\mathcal{C}^{1}$ with
$$\aligned
\langle J'(u),\varphi\rangle=&\int_{\mathbb{R}^N}(\nabla u\nabla\varphi+V(x)u\varphi)dx\\
&-\int_{\mathbb{R}^N}
\int_{\mathbb{R}^N}\frac{(|u(x)|^{2_{\mu}^{\ast}}+|u(x)|^{p})(|u(y)|^{2_{\mu}^{\ast}-2}u(y)\varphi(y)
+\frac{p}{2_{\mu}^{\ast}}|u(y)|^{p-2}u(y)\varphi(y))}
{|x-y|^{\mu}}dxdy.
\endaligned$$
So $u$ is a weak solution of \eqref{CCe1} if and only if $u$ is a critical point of the functional $J$.

\subsection{Concentration-compactness principle}
To describe the lack of compactness of the injection from $D^{1,2}(\mathbb{R}^N)$ to $L^{2^*}(\mathbb{R}^N)$,  P.L. Lions established the well known Concentration-compactness principles  \cite{Ls1, Ls2, Ls3, Ls4}. Here we would like to recall the second concentration-compactness principle \cite{Ls1} for the convenience of the readers.
\begin{lem}\label{Concentration-compactness principle}
Let $\{u_{n}\}$ be a bounded sequence in $D^{1,2}(\mathbb{R}^N)$ converging weakly and a.e. to some $u_0\in D^{1,2}(\mathbb{R}^N)$. $|\nabla u_{n}|^{2}\rightharpoonup \omega$, $|u_{n}|^{2^*}\rightharpoonup \zeta$ weakly in the sense of measures where $\omega$ and $\zeta$ are bounded non-negative measures on $\mathbb{R}^N$. Then we have:\\
(1) there exists some at most countable set $I$, a family $\{z_i:i\in I\}$ of distinct points in $\mathbb{R}^N$, and a family $\{\zeta_i:i\in I\}$ of positive numbers such that
$$
\zeta=|u_0|^{2^*}+\sum_{i\in I}\zeta_i\delta_{z_i},
$$
where $\delta_{x}$ is the Dirac-mass of mass 1 concentrated at $x\in\mathbb{R}^N$.\\
(2) In addition we have
$$
\omega\geq|\nabla u_0|^{2}+\sum_{i\in I}\omega_i\delta_{z_i}
$$
for some family $\{\omega_i:i\in I\}$, $\omega_i>0$ satisfying
$$
S\zeta_i^{\frac{2}{2^*}}\leq\omega_i,\ \ \mbox{for all } i\in I.
$$
In particular, $\sum_{i\in I}\zeta_i^{\frac{2}{2^*}}<\infty$.
\end{lem}
The second concentration-compactness principle, roughly
speaking, is only concerned with a possible concentration of a weakly convergent sequence at finite points and it does not provide any information about the loss
of mass of a sequence at infinity. The following concentration-compactness principle at infinity was developed by Chabrowski\cite{C}, J. Bianchi, Chabrowski, Szulkin \cite{BCS}, Ben-Naoum, Troestler, Willem \cite{BTW} which provided some quantitative information about the loss of mass of a sequence at infinity.

\begin{lem}\label{Concentration-compactness principle2}
Let $\{u_{n}\}\subset D^{1,2}(\mathbb{R}^N)$ be a sequence in Lemma \ref{Concentration-compactness principle} and define
$$
\omega_{\infty}:=\lim_{R\rightarrow\infty}\overline{\lim}_{n\rightarrow\infty}\int_{|x|\geq R}|\nabla u_{n}|^{2}dx,\ \ \
\zeta_{\infty}:=\lim_{R\rightarrow\infty}\overline{\lim}_{n\rightarrow\infty}\int_{|x|\geq R}| u_{n}|^{2^{\ast}}dx.
$$
Then it follows that
$$
S\zeta_{\infty}^{\frac{2}{2^{\ast}}}\leq \omega_{\infty},
$$
$$
\overline{\lim}_{n\rightarrow\infty}|\nabla u_{n}|_{2}^{2}=\int_{\mathbb{R}^N}d\omega+\omega_{\infty},
$$
$$
\overline{\lim}_{n\rightarrow\infty}|u_{n}|_{2^{\ast}}^{2^{\ast}}=\int_{\mathbb{R}^N}d\zeta+\zeta_{\infty}.
$$
\end{lem}

 The concentration-compactness principles \cite{Ls1, Ls2, Ls3, Ls4}
 help not only to investigate
 the behavior of the weakly convergent sequences in Sobolev spaces where
the lack of compactness occurs either due to the appearance of a critical Sobolev
exponent or due to the unboundedness of a domain and but also to find level sets of a given variational functional for which the
Palais-Smale condition holds. It was mentioned in the famous paper by P.L. Lions \cite{Ls1} that the limit embeddings
$$
\left(\int_{\mathbb{R}^N}\int_{\mathbb{R}^N}\frac{|u(x)|^{2_{\mu}^{\ast}}|u(y)|^{2_{\mu}^{\ast}}}
{|x-y|^{\mu}}dxdy\right)^{\frac{1}{2_{\mu}^{\ast}}}\leq C_0\int_{\mathbb{R}^N}|\nabla u|^{2}dx
$$
also cause the concentration of a weakly convergent sequence at finite points and the results in Lemma \ref{Concentration-compactness principle} holds with $|u_n|^{2^{\ast}}$ replaced by
$$
|u_n|^{2_{\mu}^{\ast}}\int_{\mathbb{R}^N}\frac{|u_n(y)|^{2_{\mu}^{\ast}}}
{|x-y|^{\mu}}dy.
$$
Moreover, a version of concentration-compactness principle corresponding to Lemma \ref{Concentration-compactness principle} was established in \cite{Ls2} to study the minimizing problem associated to the attainability of the best constant in the Hardy-Littlewood-Sobolev inequality of the form
$$
|\frac{1}{|x|^{\mu}}\ast u|_q\leq C_0 |u|_p
$$
 for some $C_0$ depending on $N, \mu, q, p$ where $0<\mu< N$ and $p, q$ satisfy
 $$
\frac1p+\frac{\mu}{n}=1+\frac1q.
 $$

In the present paper we are interested in the existence of solutions for the critical Choquard equation due to the Hardy--Littlewood--Sobolev inequality. Since the lack of compactness also occurs when people considers the critical Choquard equation in unbounded domain,  it is quite natural for people to turn to a possible use of the second concentration-compactness principle involving the convolution type nonlinearities. However, to the best knowledge of the authors, there seems no such existing lemmas that describe the
 possible concentration of a weakly convergent sequence both at finite points and at infinity.
 And there also seems no application of such a second concentration-compactness principle in studying the critical Choquard equation. Although the main idea is taken from \cite{Ls1, Ls2}, we would like to give a proof of it for readers's convenience.

\begin{lem}\label{CCP1} Let $\{u_{n}\}$ be a bounded sequence in $D^{1,2}(\mathbb{R}^N)$ converging weakly and a.e. to some $u_{0}$ and $\omega,  \omega_{\infty}, \zeta, \zeta_{\infty}$ be the bounded nonnegative measures in Lemma \ref{Concentration-compactness principle}  and Lemma \ref{Concentration-compactness principle2}. Assume that $$\Big(\displaystyle\int_{\mathbb{R}^N}\frac{|u_{n}(y)|^{2_{\mu}^{\ast}}}{|x-y|^{\mu}}dy\Big)| u_{n}(x)|^{2_{\mu}^{\ast}}\rightharpoonup \nu$$ weakly in the sense of measure where $\nu$ is a bounded positive measure on $\mathbb{R}^N$ and define
$$\aligned
\nu_{\infty}&:=\lim_{R\rightarrow\infty}\overline{\lim}_{n\rightarrow\infty}\int_{|x|\geq R}\Big(\int_{\mathbb{R}^N}\frac{|u_{n}(y)|^{2_{\mu}^{\ast}}}{|x-y|^{\mu}}dy\Big)| u_{n}(x)|^{2_{\mu}^{\ast}}dx.
\endaligned$$
Then, there exists a countable sequence of points $\{z_{i}\}_{i\in I}\subset \mathbb{R}^N $ and families  of positive numbers $\{\nu_i:i\in I\}$ , $\{\zeta_i:i\in I\}$ and $\{\omega_i:i\in I\}$ such that
\begin{equation}\label{cp5}
\nu=\Big(\displaystyle\int_{\mathbb{R}^N}\frac{|u_0(y)|^{2_{\mu}^{\ast}}}{|x-y|^{\mu}}dy\Big)| u_0(x)|^{2_{\mu}^{\ast}}+ \Sigma_{i\in I}\nu_{i}\delta_{z_{i}},\ \  \Sigma_{i\in I}\nu_{i}^{\frac{1}{2_{\mu}^{\ast}}}<\infty,
\end{equation}
\begin{equation}\label{cp51}
\omega\geq|\nabla u_0|^{2}+\sum_{i\in I}\omega_i\delta_{z_i},
\end{equation}
\begin{equation}\label{cp52}
\zeta\geq|u_0|^{2^*}+\sum_{i\in I}\zeta_i\delta_{z_i},
\end{equation}
and
\begin{equation}\label{cp6}
 \ S_{H,L}\nu_{i}^{\frac{1}{2_{\mu}^{\ast}}}\leq\omega_{i},\ \  \nu_{i}^{\frac{N}{2N-\mu}}\leq C(N,\mu)^{\frac{N}{2N-\mu}}\zeta_{i},
\end{equation}
where $\delta_{x}$ is the Dirac-mass of mass 1 concentrated at $x\in\mathbb{R}^N$.

For the energy at infinity, we have
\begin{equation}\label{cp51}
\overline{\lim}_{n\rightarrow\infty}\int_{\mathbb{R}^N}\int_{\mathbb{R}^N}
\frac{|u_{n}(y)|^{2_{\mu}^{\ast}}|u_{n}(x)|^{2_{\mu}^{\ast}}}{|x-y|^{\mu}}
dydx=\nu_{\infty}+\int_{\mathbb{R}^N}d\nu,
\end{equation}
and
\begin{equation}\label{cp71}
C(N,\mu)^{\frac{-2N}{2N-\mu}}\nu_{\infty}^{\frac{2N}{2N-\mu}}\leq \zeta_{\infty}(\int_{\mathbb{R}^N}d\zeta+\zeta_{\infty}),\  \  S_{H,L}^{2}\nu_{\infty}^{\frac{2}{2_{\mu}^{\ast}}}\leq \omega_{\infty}(\int_{\mathbb{R}^N}d\omega+\omega_{\infty}).
\end{equation}
Moreover, if $u=0$ and $\displaystyle\int_{\mathbb{R}^N}d\omega
=S_{H,L}\left(\int_{\mathbb{R}^N}d\nu\right)^{\frac{1}{2_{\mu}^{\ast}}}$, then $\nu$ is concentrated at a single point.
\end{lem}
\begin{proof}
Since $\{u_{n}\}$ is a bounded sequence in $D^{1,2}(\mathbb{R}^N)$ converging weakly to $u$, denote by $v_{n}:=u_{n}-u_0$, we have $v_{n}(x)\rightarrow0$ a.e. in $\mathbb{R}^N$  and $v_{n}$ converges weakly to $0$ in $D^{1,2}(\R^N)$. Applying Lemma \ref{BLN}, in the sense of measure, we have
\begin{center}
$|\nabla v_{n}|^{2}\rightharpoonup\varpi:=\omega-|\nabla u_0|^{2}$,
\end{center}
\begin{center}
$\Big(\displaystyle\int_{\mathbb{R}^N}\frac{|v_{n}(y)|^{2_{\mu}^{\ast}}}{|x-y|^{\mu}}dy\Big)| v_{n}(x)|^{2_{\mu}^{\ast}}\rightharpoonup\kappa:=\nu-\Big(\displaystyle\int_{\mathbb{R}^N}\frac{|u_0(y)|^{2_{\mu}^{\ast}}}{|x-y|^{\mu}}dy\Big)| u_0(x)|^{2_{\mu}^{\ast}}$
\end{center}
\begin{center}
 and $|v_{n}|^{2^{\ast}}\rightharpoonup\varsigma:=\zeta-|u_0|^{2^{\ast}}$.
 \end{center}

To prove the possible concentration at finite points, we first show that
\begin{equation}\label{cp0}
\Big|\int_{\mathbb{R}^N}(|x|^{-\mu}\ast |\phi v_{n}(x)|^{2_{\mu}^{\ast}})| \phi v_{n}(x)|^{2_{\mu}^{\ast}}dx-\int_{\mathbb{R}^N}(|x|^{-\mu}\ast |v_{n}(x)|^{2_{\mu}^{\ast}})|\phi (x)|^{2_{\mu}^{\ast}}|\phi v_{n}(x)|^{2_{\mu}^{\ast}}dx\Big|\rightarrow0,
\end{equation}
where $\phi\in \mathcal{C}_{0}^{\infty}(\mathbb{R}^N)$. In fact, we denote
$$
\Phi_{n}(x):=(|x|^{-\mu}\ast |\phi v_{n}(x)|^{2_{\mu}^{\ast}})|\phi v_{n}(x)|^{2_{\mu}^{\ast}}-
(|x|^{-\mu}\ast | v_{n}(x)|^{2_{\mu}^{\ast}})|\phi (x)|^{2_{\mu}^{\ast}}|\phi v_{n}(x)|^{2_{\mu}^{\ast}}.
$$
Since $\phi\in \mathcal{C}_{0}^{\infty}(\mathbb{R}^N)$, we have for every $\delta>0$ there exists $M>0$ such that
\begin{equation}\label{cp00}
\int_{|x|\geq M}|\Phi_{n}(x)|dx<\delta \ \ \ (\forall n\geq1).
\end{equation}
Since the Riesz potential defines a linear operator, from the fact that $v_n(x)\to 0$ a.e. in $\mathbb{R}^N$ we know that
$$
\int_{\mathbb{R}^N}\frac{|v_{n}(y)|^{2_{\mu}^{\ast}}}{|x-y|^{\mu}}dy\rightarrow 0
$$
a.e. in $\mathbb{R}^N$ and so we have $\Phi_{n}(x)\rightarrow0$ a.e. in $\mathbb{R}^N$. Notice that
$$
\aligned
\Phi_{n}(x)&=\int_{\mathbb{R}^N}
\frac{(|\phi (y)|^{2_{\mu}^{\ast}}-|\phi (x)|^{2_{\mu}^{\ast}})|v_{n}(y)|^{2_{\mu}^{\ast}}}{|x-y|^{\mu}}dy|\phi v_{n}(x)|^{2_{\mu}^{\ast}}\\
:&=\int_{\mathbb{R}^N}
L(x,y)|v_{n}(y)|^{2_{\mu}^{\ast}}dy|\phi v_{n}(x)|^{2_{\mu}^{\ast}}.
\endaligned
$$
For almost all $x$, there exists $R>0$ large enough such that
$$
\int_{\mathbb{R}^N}
L(x,y)|v_{n}(y)|^{2_{\mu}^{\ast}}dy
=\int_{|y|\leq R}
L(x,y)|v_{n}(y)|^{2_{\mu}^{\ast}}dy
-|\phi (x)|^{2_{\mu}^{\ast}}\int_{|y|\geq R}
\frac{|v_{n}(y)|^{2_{\mu}^{\ast}}}{|x-y|^{\mu}}dy.
$$
 As observed in \cite{Ls2} that $L(x,y)\in L^{r}(B_{R})$ for each $x$, where $r<\frac{N}{\mu-1}$ if $\mu>1$, $r\leq+\infty$ if $0<\mu\leq1$. By the Young inequality, there exists $s>\frac{2N}{\mu}$ such that $$
\Big(\int_{B_{M}}\Big(\int_{B_{R}}
L(x,y)|v_{n}(y)|^{2_{\mu}^{\ast}}dy\Big)^{s}dx\Big)^{\frac{1}{s}}\leq C_{\phi}|L(x,y)|_{r}||v_{n}|^{2_{\mu}^{\ast}}|_{\frac{2N}{2N-\mu}}\leq C_{\phi}'
$$
where $M$ is given in \eqref{cp00}. It is easy to see that for $R>0$ large enough
$$
\Big(\int_{B_{M}}\Big(
|\phi (x)|^{2_{\mu}^{\ast}}\int_{|y|\geq R}
\frac{|v_{n}(y)|^{2_{\mu}^{\ast}}}{|x-y|^{\mu}}dy\Big)^{s}dx\Big)^{\frac{1}{s}}\leq C
$$
and so, we have
$$
\Big(\int_{B_{M}}\Big(\int_{\mathbb{R}^N}
L(x,y)|v_{n}(y)|^{2_{\mu}^{\ast}}dy\Big)^{s}dx\Big)^{\frac{1}{s}}\leq C_{\phi}''.
$$
Then, we can get for $\tau>0$ small enough
$$
\int_{B_{M}}|\Phi_{n}(x)|^{1+\tau}dx\leq \Big(\int_{B_{M}}\Big(\int_{\mathbb{R}^N}
L(x,y)|v_{n}(y)|^{2_{\mu}^{\ast}}dy\Big)^{s}dx\Big)^{\frac{1}{s}}
\Big(\int_{B_{M}}|\phi v_{n}|^{2^{\ast}}dx\Big)^{\frac{2N-\mu}{2N}}\leq C_{\phi}''.
$$
Combining this and $\Phi_{n}(x)\rightarrow0$ a.e. in $\mathbb{R}^N$, we can get
$$
\int_{B_{M}}|\Phi_{n}(x)|dx\rightarrow0 \ \ \ (n\rightarrow\infty).
$$
By this and \eqref{cp00},  we have
$$
\int_{\mathbb{R}^N}|\Phi_{n}(x)|dx\rightarrow0.
$$

 For all $\phi\in \mathcal{C}_{0}^{\infty}(\mathbb{R}^N)$, by the Hardy--Littlewood--Sobolev inequality, we have
$$
\int_{\mathbb{R}^N}\Big(\int_{\mathbb{R}^N}\frac{|\phi v_{n}(y)|^{2_{\mu}^{\ast}}}{|x-y|^{\mu}}
dy\Big)|\phi v_{n}(x)|^{2_{\mu}^{\ast}}dx
\leq C(N,\mu)|\phi v_{n}|_{2^{\ast}}^{2\cdot2_{\mu}^{\ast}}.
$$
By \eqref{cp0}, we have
$$
\int_{\mathbb{R}^N}|\phi(x)|^{2\cdot2_{\mu}^{\ast}}\Big(\int_{\mathbb{R}^N}\frac{|v_{n}(y)|^{2_{\mu}^{\ast}}}{|x-y|^{\mu}}
dy\Big)| v_{n}(x)|^{2_{\mu}^{\ast}}dx
\leq C(N,\mu)|\phi v_{n}|_{2^{\ast}}^{2\cdot2_{\mu}^{\ast}}+o(1).
$$
Passing to the limit as $n\rightarrow+\infty$ we obtain
\begin{equation}\label{cp9}
\int_{\mathbb{R}^N}|\phi(x)|^{2\cdot2_{\mu}^{\ast}}d\kappa
\leq C(N,\mu)\Big(\int_{\mathbb{R}^N}|\phi|^{2^{\ast}}d\varsigma\Big)^{\frac{2N-\mu}{N}}.
\end{equation}
Applying Lemma 1.2 in \cite{Ls1} we know \eqref{cp52} holds.

Taking $\phi=\chi_{\{z_{i}\}}$, $i\in I$, in \eqref{cp9}, we get
$$
\nu_{i}^{\frac{N}{2N-\mu}}\leq C(N,\mu)^{\frac{N}{2N-\mu}}\zeta_{i}, \ \forall i \in I.
$$
By the definition of $S_{H,L}$, we also have
$$
\left(\int_{\mathbb{R}^N}\Big(\int_{\mathbb{R}^N}\frac{|\phi v_{n}(y)|^{2_{\mu}^{\ast}}}{|x-y|^{\mu}}dy\Big)| \phi v_{n}(x)|^{2_{\mu}^{\ast}}dx\right)^{\frac{N-2}{2N-\mu}}
S_{H,L}\leq\displaystyle\int_{\mathbb{R}^N}|\nabla (\phi v_{n})|^{2}dx.
$$
By \eqref{cp0} and $v_{n}\rightarrow0$ in $L_{loc}^{2}(\mathbb{R}^N)$, we have
$$
\left(\int_{\mathbb{R}^N}|\phi(x)|^{2\cdot2_{\mu}^{\ast}}\Big(\int_{\mathbb{R}^N}
\frac{|v_{n}(y)|^{2_{\mu}^{\ast}}}{|x-y|^{\mu}}dy\Big)| v_{n}(x)|^{2_{\mu}^{\ast}}dx\right)^{\frac{N-2}{2N-\mu}}
S_{H,L}\leq\displaystyle\int_{\mathbb{R}^N}\phi^{2}|\nabla v_{n}|^{2}dx+o(1).
$$
Passing to the limit as $n\rightarrow+\infty$ we obtain
\begin{equation}\label{cp8}
\left(\int_{\mathbb{R}^N}|\phi(x)|^{2\cdot2_{\mu}^{\ast}}d\kappa\right)^{\frac{N-2}{2N-\mu}}
S_{H,L}\leq\displaystyle\int_{\mathbb{R}^N}\phi^{2}d\varpi.
\end{equation}
Applying Lemma 1.2 in \cite{Ls1} again we know \eqref{cp51} holds.
Now by taking $\phi=\chi_{\{z_{i}\}}$, $i\in I$, in \eqref{cp8}, we get
$$
S_{H,L}\nu_{i}^{\frac{1}{2_{\mu}^{\ast}}}\leq\omega_{i}, \ \forall i \in I.
$$
Thus we proved \eqref{cp5} and \eqref{cp6}.

Next we are going to prove the possible loss of mass at infinity. For $R>1$, let $\psi_{R}\in \mathcal{C}^{\infty}(\mathbb{R}^N)$ be such that $\psi_{R}(x)=1$ for $|x|>R+1$, $\psi_{R}(x)=0$ for $|x|<R$ and $0\leq\psi_{R}(x)\leq1$ on $\mathbb{R}^N$. For every $R>1$, we have
$$\aligned
&\overline{\lim}_{n\rightarrow\infty}\int_{\mathbb{R}^N}\int_{\mathbb{R}^N}
\frac{|v_{n}(y)|^{2_{\mu}^{\ast}}|v_{n}(x)|^{2_{\mu}^{\ast}}}{|x-y|^{\mu}}
dydx\\
&=\overline{\lim}_{n\rightarrow\infty}\int_{\mathbb{R}^N}\int_{\mathbb{R}^N}
\frac{|u_{n}(y)|^{2_{\mu}^{\ast}}|u_{n}(x)|^{2_{\mu}^{\ast}}}{|x-y|^{\mu}}
dydx-\int_{\mathbb{R}^N}\int_{\mathbb{R}^N}
\frac{|u_0(y)|^{2_{\mu}^{\ast}}|u_0(x)|^{2_{\mu}^{\ast}}}{|x-y|^{\mu}}
dydx\\
&=\overline{\lim}_{n\rightarrow\infty}\Big(\int_{\mathbb{R}^N}\int_{\mathbb{R}^N}
\frac{|u_{n}(y)|^{2_{\mu}^{\ast}}|u_{n}(x)|^{2_{\mu}^{\ast}}\psi_{R}(x)}{|x-y|^{\mu}}
dydx+\int_{\mathbb{R}^N}\int_{\mathbb{R}^N}
\frac{|u_{n}(y)|^{2_{\mu}^{\ast}}|u_{n}(x)|^{2_{\mu}^{\ast}}(1-\psi_{R}(x))}{|x-y|^{\mu}}
dydx\Big)\\
&\hspace{1cm}-\int_{\mathbb{R}^N}\int_{\mathbb{R}^N}
\frac{|u_0(y)|^{2_{\mu}^{\ast}}|u_0(x)|^{2_{\mu}^{\ast}}}{|x-y|^{\mu}}
dydx\\
&=\overline{\lim}_{n\rightarrow\infty}\int_{\mathbb{R}^N}\int_{\mathbb{R}^N}
\frac{|u_{n}(y)|^{2_{\mu}^{\ast}}|u_{n}(x)|^{2_{\mu}^{\ast}}\psi_{R}(x)}{|x-y|^{\mu}}
dydx+\int_{\mathbb{R}^N}(1-\psi_{R})d\nu\\
&\hspace{1cm}+\int_{\mathbb{R}^N}\int_{\mathbb{R}^N}
\frac{|u_0(y)|^{2_{\mu}^{\ast}}|u_0(x)|^{2_{\mu}^{\ast}}(1-\psi_{R}(x))}{|x-y|^{\mu}}
dydx-\int_{\mathbb{R}^N}\int_{\mathbb{R}^N}
\frac{|u_0(y)|^{2_{\mu}^{\ast}}|u_0(x)|^{2_{\mu}^{\ast}}}{|x-y|^{\mu}}
dydx.
\endaligned$$
When $R\rightarrow\infty$, we obtain, by Lebesgue's theorem,
$$
\overline{\lim}_{n\rightarrow\infty}\int_{\mathbb{R}^N}\int_{\mathbb{R}^N}
\frac{|u_{n}(y)|^{2_{\mu}^{\ast}}|u_{n}(x)|^{2_{\mu}^{\ast}}}{|x-y|^{\mu}}
dydx=\nu_{\infty}+\int_{\mathbb{R}^N}d\nu.
$$
By the Hardy--Littlewood--Sobolev inequality, we have
$$\aligned
\nu_{\infty}&=\lim_{R\rightarrow\infty}
\overline{\lim}_{n\rightarrow\infty}
\int_{\mathbb{R}^N}
\Big(\int_{\mathbb{R}^N}
\frac{|u_{n}(y)|^{2_{\mu}^{\ast}}}{|x-y|^{\mu}}dy\Big)|\psi_{R} u_{n}(x)|^{2_{\mu}^{\ast}}dx\\
&\leq C(N,\mu)\lim_{R\rightarrow\infty}\overline{\lim}_{n\rightarrow\infty}
\Big(\int_{\mathbb{R}^N}|u_{n}|^{2^{\ast}}dx
\int_{\mathbb{R}^N}|\psi_{R}u_{n}|^{2^{\ast}}dx\Big)^{\frac{2N-\mu}{2N}}\\
&= C(N,\mu)\Big(\zeta_{\infty}(\int_{\mathbb{R}^N}d\zeta+\zeta_{\infty})\Big)^{\frac{2N-\mu}{2N}},
\endaligned$$
which means
$$
C(N,\mu)^{\frac{-2N}{2N-\mu}}\nu_{\infty}^{\frac{2N}{2N-\mu}}\leq \zeta_{\infty}(\int_{\mathbb{R}^N}d\zeta+\zeta_{\infty}).
$$
Similarly, by the definition of $S_{H,L}$ and $\nu_{\infty}$, we have
$$
\aligned
\nu_{\infty}&=\lim_{R\rightarrow\infty}
\overline{\lim}_{n\rightarrow\infty}
\int_{\mathbb{R}^N}
\Big(\int_{\mathbb{R}^N}
\frac{|u_{n}(y)|^{2_{\mu}^{\ast}}}{|x-y|^{\mu}}dy\Big)| \psi_{R}u_{n}(x)|^{2_{\mu}^{\ast}}dx\\
&\leq C(N,\mu)\lim_{R\rightarrow\infty}\overline{\lim}_{n\rightarrow\infty}
\Big(\int_{\mathbb{R}^N}|u_{n}|^{2^{\ast}}dx
\int_{\mathbb{R}^N}|\psi_{R}u_{n}|^{2^{\ast}}dx\Big)^{\frac{2N-\mu}{2N}}\\
&\leq S_{H,L}^{-2_{\mu}^{\ast}}\lim_{R\rightarrow\infty}\overline{\lim}_{n\rightarrow\infty}
\Big(\int_{\mathbb{R}^N}|\nabla u_{n}|^{2}dx
\int_{\mathbb{R}^N}|\nabla(\psi_{R}u_{n})|^{2}dx\Big)^{\frac{2_{\mu}^{\ast}}{2}}\\
&= S_{H,L}^{-2_{\mu}^{\ast}}\Big(\omega_{\infty}(\int_{\mathbb{R}^N}d\omega+\omega_{\infty})\Big)^{\frac{2_{\mu}^{\ast}}{2}},
\endaligned$$
which means
$$
 S_{H,L}^{2}\nu_{\infty}^{\frac{2}{2_{\mu}^{\ast}}}\leq \omega_{\infty}(\int_{\mathbb{R}^N}d\omega+\omega_{\infty}).
$$

Moreover, if $u=0$ then $\kappa=\nu$ and $\varpi=\omega$.
Then the H\"{o}lder inequality and \eqref{cp8} imply that, for $\phi\in \mathcal{C}_{0}^{\infty}(\mathbb{R}^N)$,
$$
\left(\int_{\mathbb{R}^N}|\phi(x)|^{2\cdot2_{\mu}^{\ast}}d\nu\right)^{\frac{N-2}{2N-\mu}}
S_{H,L}\leq\left(\int_{\mathbb{R}^N}d\omega\right)^{\frac{N-\mu+2}{2N-\mu}}
\left(\int_{\mathbb{R}^N}\phi^{2\cdot2_{\mu}^{\ast}}d\omega\right)^{\frac{N-2}{2N-\mu}}.
$$
Thus we can deduce that $\nu=S_{H,L}^{-2_{\mu}^{\ast}}
\left(\displaystyle\int_{\mathbb{R}^N}d\omega\right)^{\frac{N-\mu+2}{N-2}}\omega$. It follows from \eqref{cp8} that, for $\phi\in \mathcal{C}_{0}^{\infty}(\mathbb{R}^N)$,
$$
\left(\int_{\mathbb{R}^N}|\phi(x)|^{2\cdot2_{\mu}^{\ast}}d\nu\right)^{\frac{N-2}{2N-\mu}}
\left(\int_{\mathbb{R}^N}d\nu\right)^{\frac{N-\mu+2}{2N-\mu}}\leq
\int_{\mathbb{R}^N}|\phi|^{2}d\nu.
$$
And so, for each open set $\Omega$,
$$
\nu(\Omega)^{\frac{N-2}{2N-\mu}}\nu(\mathbb{R}^N)^{\frac{N-\mu+2}{2N-\mu}}
\leq\nu(\Omega).
$$
It follows that $\nu$ is concentrated at a single point.
\end{proof}

\subsection{Convergence of $(PS)$ sequences}

 Let $\{u_n\}$ be a $(PS)$ sequence of $J$ at level $c$,  it is easy to see that $\{u_n\}$ is bounded in $E$. Hence, without loss of generality, we may assume that $\{u_n\}$ converges weakly and a.e. to some $u_0\in E$. Then we are able to recover the lack of compactness by applying the second concentration-compactness principle to the nonlocal Choquard equation. In fact we have the following proposition which was inspired by \cite{Z}.

\begin{Prop}\label{F1}
There exists a positive number $c_0>0$ such that every $(PS)_c$ sequence $\{u_n\}$ of $J$ with $c<c_0$ satisfies
$$
\lim_{n\rightarrow\infty}\int_{\mathbb{R}^N}\int_{\mathbb{R}^N}\frac{|(u_{n}-u_{0})(x)|^{2_{\mu}^{\ast}}
|(u_{n}-u_{0})(y)|^{2_{\mu}^{\ast}}}
{|x-y|^{\mu}}dxdy=0,
$$
where $u_0\in E$ is the weak limit of $\{u_n\}$.
\end{Prop}
\begin{proof}
Let $\eta\in \mathcal{C}_{0}^{\infty}([0,\infty))$ be a standard cut-off function on $[0, 1]$, that is,
$$
\eta(t)\equiv1,\ \ t\in[0, 1];\ \ \eta(t)\equiv0,\ \ t>2;\ \
|\eta'(t)|\leq C,\ \ 0\leq\eta(t)\leq1
$$
for some $C > 0$. Fix $i\in I$. For $\varepsilon> 0$, put
\begin{equation}\label{CUT1}
\eta_{\varepsilon}=\eta(\frac{x-z_{i}}{\varepsilon}),\ \ B=B_{2\varepsilon}(z_{i}).
\end{equation}
It follows from the H\"{o}lder inequality and the Sobolev inequality that for all $\sigma\in[0,2^{\ast})$,
$$
\aligned
\int_{B}|u_n|^{\sigma}dx\leq|B|^{1-\frac{\sigma}{2^{\ast}}}
\Big(\int_{B}|u_n|^{2^{\ast}}dx\Big)^{\frac{\sigma}{2^{\ast}}}
\leq C|B|^{1-\frac{\sigma}{2^{\ast}}}
\Big(\int_{B}|\nabla u_n|^{2}dx\Big)^{\frac{\sigma}{2}}
=o(1), \ \hbox{as}\ \ \varepsilon\rightarrow0^{+}.
\endaligned
$$
Hence, by the Hardy--Littlewood--Sobolev inequality, as $\varepsilon\rightarrow0^{+}$, there holds
\begin{equation}\label{p91}
\int_{\mathbb{R}^N}
\int_{\mathbb{R}^N}\frac{|u_n(x)|^{2_{\mu}^{\ast}}
|u_n(y)|^{p}\eta_{\varepsilon}(y)}
{|x-y|^{\mu}}dxdy=o(1),
\end{equation}
\begin{equation}\label{p92}
\int_{\mathbb{R}^N}
\int_{\mathbb{R}^N}\frac{|u_n(x)|^{p}
|u_n(y)|^{p}\eta_{\varepsilon}(y)}
{|x-y|^{\mu}}dxdy=o(1)
\end{equation}
and
\begin{equation}\label{p9}
\int_{\mathbb{R}^N}
\int_{\mathbb{R}^N}\frac{|u_n(x)|^{p}
|u_n(y)|^{2_{\mu}^{\ast}}\eta_{\varepsilon}(y)}
{|x-y|^{\mu}}dxdy=o(1).
\end{equation}
Using Proposition \ref{CCP1}, and passing to the limit by first letting $n\rightarrow\infty$ and then letting $\varepsilon\rightarrow0^{+}$, we have
\begin{equation}\label{p10}
\lim_{n\rightarrow\infty}\int_{\mathbb{R}^N}\int_{\mathbb{R}^N}\frac{|u_n(x)|
^{2_{\mu}^{\ast}}|u_n(y)|^{2_{\mu}^{\ast}}\eta_{\varepsilon}(y)}
{|x-y|^{\mu}}dxdy=\nu_{i}.
\end{equation}

 For $R>0$, put
\begin{equation}\label{CUT21}
\eta_{R}=\eta(\frac{2R}{|x|})
\end{equation}
and denote
$$\aligned
V_{\infty}:=\lim_{R\rightarrow\infty}\overline{\lim}_{n\rightarrow\infty}
&\int_{\mathbb{R}^N}V|u_{n}|^{2}\eta_{R}dx,\\
F_{\infty}:=\lim_{R\rightarrow\infty}\overline{\lim}_{n\rightarrow\infty}
&\Big(\frac{p}{2_{\mu}^{\ast}}\int_{\mathbb{R}^N}\int_{\mathbb{R}^N}\frac{|u_n(x)|
^{2_{\mu}^{\ast}}|u_n(y)|^{p}\eta_{R}(y)}
{|x-y|^{\mu}}dxdy\\
&\hspace{4mm}+\int_{\mathbb{R}^N}\int_{\mathbb{R}^N}\frac{|u_n(x)|
^{p}|u_n(y)|^{2_{\mu}^{\ast}}\eta_{R}(y)}
{|x-y|^{\mu}}dxdy\\
&\hspace{8mm}+\frac{p}{2_{\mu}^{\ast}}\int_{\mathbb{R}^N}\int_{\mathbb{R}^N}\frac{|u_n(x)|
^{p}|u_n(y)|^{p}\eta_{R}(y)}
{|x-y|^{\mu}}dxdy\Big).
\endaligned$$
Multiplying $J'(u_n)$ with the test function $u_n\eta_{R}$, we obtain by the definition of
$
\omega_{\infty},
\nu_{\infty}$ that
\begin{equation}\label{p13}
\omega_{\infty}+V_{\infty}=\nu_{\infty}+F_{\infty}.
\end{equation}

It follows that
\begin{equation}\label{p19}
\aligned
c+o(1)&=J(u_n)-\frac{1}{2}\langle J'(u_n),u_n\rangle\\
&=\frac{p+2_{\mu}^{\ast}-2}{2\cdot2_{\mu}^{\ast}}\int_{\mathbb{R}^N}\int_{\mathbb{R}^N}\frac{|u_n(x)|
^{2_{\mu}^{\ast}}|u_n(y)|^{p}}
{|x-y|^{\mu}}dxdy
+\frac{p-1}{2\cdot2_{\mu}^{\ast}}\int_{\mathbb{R}^N}\int_{\mathbb{R}^N}\frac{|u_n(x)|
^{p}|u_n(y)|^{p}}
{|x-y|^{\mu}}dxdy\\
&\hspace{0.5cm}+\frac{N+2-\mu}{4N-2\mu}\int_{\mathbb{R}^N}\int_{\mathbb{R}^N}\frac{|u_n(x)|
^{2_{\mu}^{\ast}}|u_n(y)|^{2_{\mu}^{\ast}}}
{|x-y|^{\mu}}dxdy\\
&\geq\alpha'\Big[\Big(\frac{p}{2_{\mu}^{\ast}}+1\Big)\int_{\mathbb{R}^N}\int_{\mathbb{R}^N}\frac{|u_n(x)|
^{2_{\mu}^{\ast}}|u_n(y)|^{p}}
{|x-y|^{\mu}}dxdy
+\frac{p}{2_{\mu}^{\ast}}\int_{\mathbb{R}^N}\int_{\mathbb{R}^N}\frac{|u_n(x)|
^{p}|u_n(y)|^{p}}
{|x-y|^{\mu}}dxdy\Big]\\
&\hspace{0.5cm}+\frac{N+2-\mu}{4N-2\mu}\int_{\mathbb{R}^N}\int_{\mathbb{R}^N}\frac{|u_n(x)|
^{2_{\mu}^{\ast}}|u_n(y)|^{2_{\mu}^{\ast}}}
{|x-y|^{\mu}}dxdy\\
&\geq\alpha'F_{\infty}
+\frac{N+2-\mu}{4N-2\mu}\int_{\mathbb{R}^N}\int_{\mathbb{R}^N}\frac{|u_0(x)|
^{2_{\mu}^{\ast}}|u_0(y)|^{2_{\mu}^{\ast}}}
{|x-y|^{\mu}}dxdy
+\frac{N+2-\mu}{4N-2\mu}(\nu_{\infty}+\sum_{i\in I}\nu_{i})+o(1)\\
&\geq\frac{N+2-\mu}{4N-2\mu}\sum_{i\in I}\nu_{i}+\alpha'(\nu_{\infty}+F_{\infty})+o(1)\\
&\geq\frac{N+2-\mu}{4N-2\mu}\sum_{i\in I}\nu_{i}+\alpha'\omega_{\infty}+o(1)
\endaligned
\end{equation}
for some $0<\alpha'<\frac{p-1}{2p}$, where we have used Lemma \ref{CCP1},
$$
\lim_{n\rightarrow\infty}
\int_{\mathbb{R}^N}\int_{\mathbb{R}^N}\frac{|u_n(x)|
^{2_{\mu}^{\ast}}|u_n(y)|^{2_{\mu}^{\ast}}}
{|x-y|^{\mu}}dxdy
=\nu_{\infty}+\sum_{i\in I}\nu_{i}+\int_{\mathbb{R}^N}\int_{\mathbb{R}^N}\frac{|u_0(x)|
^{2_{\mu}^{\ast}}|u_0(y)|^{2_{\mu}^{\ast}}}
{|x-y|^{\mu}}dxdy
$$
and the fact $\nu_{\infty}+F_{\infty}=\omega_{\infty}+V_{\infty}\geq\omega_{\infty}$.

Now we want to show that there exists  $c_0>0$ such that if  $c<c_0$ then
the singular part and escaping part of the energy of the $(PS)_c$ sequence $\{u_n\}$ are trivial. First we claim that
\begin{equation}\label{p18}
I=\emptyset.
\end{equation}

On the contrary, assume that $I\neq\emptyset$, then there holds
\begin{equation}\label{p2}
\nu_{i}\geq S_{H,L}^{\frac{2N-\mu}{N+2-\mu}},\ i\in I.
\end{equation}
In particular, the set $I$ is finite. In fact, let $\eta_{\varepsilon}$ be the cut-off function defined in \eqref{CUT1}.
By definition, a direct computation yields
$$
\|u_n\eta_{\varepsilon}\|_{V}=\Big(\int_{\mathbb{R}^N}|\nabla (u_n\eta_{\varepsilon})|^{2}dx+\int_{\mathbb{R}^N}V_{+}|u_n
\eta_{\varepsilon}|^{2}dx\Big)^{\frac{1}{2}}\leq C\|u_n\|_{V}=O(1).
$$
Apply $J'(u_n)$ to the test function $u_n\eta_{\varepsilon}$ to obtain
\begin{equation}\label{p3}
\aligned
o(1)=\langle J'(u_n),u_n\eta_{\varepsilon}\rangle=&\int_{\mathbb{R}^N}|\nabla u_n|^{2}\eta_{\varepsilon}dx+\int_{\mathbb{R}^N}u_n\nabla u_n\nabla\eta_{\varepsilon}dx+\int_{\mathbb{R}^N}Vu_n^{2}\eta_{\varepsilon}dx\\
&-\int_{\mathbb{R}^N}
\int_{\mathbb{R}^N}\frac{(|u_n(x)|^{2_{\mu}^{\ast}}+|u_n(x)|^{p})(|u_n(y)|^{2_{\mu}^{\ast}}\eta_{\varepsilon}(y)
+\frac{p}{2_{\mu}^{\ast}}|u_n(y)|^{p}\eta_{\varepsilon}(y))}
{|x-y|^{\mu}}dxdy,
\endaligned
\end{equation}
since $\{u_n\}$ is a $(PS)$ sequence. By Lemma \ref{CCP1}, we know
\begin{equation}\label{p5}
\int_{\mathbb{R}^N}|\nabla u_n|^{2}\eta_{\varepsilon}dx\rightarrow\omega_{i}'\geq\omega_{i}.
\end{equation}
By Lemma 3.1 in \cite{Z}, we have as $\varepsilon\rightarrow0^{+}$
\begin{equation}\label{p7}
\int_{\mathbb{R}^N}u_n\nabla u_n\nabla\eta_{\varepsilon}dx=o(1)
\end{equation}
and
\begin{equation}\label{p8}
\Big|\int_{\mathbb{R}^N}Vu_n^{2}\eta_{\varepsilon}dx\Big|=o(1).
\end{equation}
From \eqref{p91}-\eqref{p9} and \eqref{p5}-\eqref{p8}, we infer that for each fixed $i\in I$
$$
\omega_{i}-\nu_{i}\leq0.
$$
Utilizing \eqref{cp6}, we finally arrive at
$$
S_{H,L}\nu_{i}^{\frac{1}{2_{\mu}^{\ast}}}-\nu_{i}\leq0.
$$
Thus \eqref{p2} follows. Now \eqref{p19} leads to a contradiction if $c_0\leq\frac{N+2-\mu}{4N-2\mu}S_{H,L}^{\frac{2N-\mu}{N+2-\mu}}$ and thus the singular part is empty.

Next we prove that
\begin{equation}\label{p18}
\nu_{\infty}=\omega_{\infty}=F_{\infty}=V_{\infty}=0.
\end{equation}
To prove that the escaping part is trivial, let
$$
a_{p}:=\frac{2^{\ast}-\frac{2Np}{2N-\mu}}{2^{\ast}-2} \ \ \mbox{and} \ \ b_{p}:=\frac{\frac{2Np}{2N-\mu}-2}{2^{\ast}-2},
$$
then $a_{p},b_{p}\in(0,1)$ and $a_{p}+b_{p}=1$.
By Lemma 3.2 in \cite{Z} we know that
\begin{equation}\label{p14}
V_{\infty}=\lim_{R\rightarrow\infty}\lim_{n\rightarrow\infty}
\int_{\mathbb{R}^N}V|u_{n}|^{2}\eta_{R}dx\geq\tau_{0}\lim_{R\rightarrow\infty}\lim_{n\rightarrow\infty}
\int_{\mathbb{R}^N}|u_{n}|^{2}\eta_{R}dx.
\end{equation}
With this fact, applying the Hardy--Littlewood--Sobolev and H\"{o}lder inequalities, we have
\begin{equation}\label{p151}
\aligned
\lim_{R\rightarrow\infty}
&\overline{\lim}_{n\rightarrow\infty}\int_{\mathbb{R}^N}\int_{\mathbb{R}^N}\frac{|u_n(x)|
^{2_{\mu}^{\ast}}|u_n(y)|^{p}\eta_{R}(y)}
{|x-y|^{\mu}}dxdy\\
&\leq C(N,\mu)\lim_{R\rightarrow\infty}\overline{\lim}_{n\rightarrow\infty}|u_n|_{2^{\ast}}^{2_{\mu}^{\ast}}||u_n|^{p} \eta_{R}|_{\frac{2N}{2N-\mu}}\\
&\leq C_1\lim_{R\rightarrow\infty}\overline{\lim}_{n\rightarrow\infty}\Big(\int_{\mathbb{R}^N}|u_{n}|^{2}\eta_{R}dx\Big)^{a_{p}}
\Big(\int_{\mathbb{R}^N}|u_{n}|^{2^{\ast}}\eta_{R}dx\Big)^{b_{p}}\\
&\leq C_1(\frac{V_{\infty}}{\tau_{0}})^{a_{p}}
\zeta_{\infty}^{b_{p}},
\endaligned
\end{equation}
\begin{equation}\label{p152}
\aligned
\lim_{R\rightarrow\infty}&\overline{\lim}_{n\rightarrow\infty}\int_{\mathbb{R}^N}\int_{\mathbb{R}^N}\frac{|u_n(x)|
^{p}|u_n(y)|^{2_{\mu}^{\ast}}\eta_{R}(y)}
{|x-y|^{\mu}}dxdy\\
&\leq C(N,\mu)\lim_{R\rightarrow\infty}\overline{\lim}_{n\rightarrow\infty}|u_n|_{\frac{2Np}{2N-\mu}}^{p}
\Big(\int_{\mathbb{R}^N}|u_{n}|^{2^{\ast}}|\eta_{R}|^{\frac{2N}{2N-\mu}}dx\Big)^{\frac{2N-\mu}{2N}}\\
&\leq C_2
\zeta_{\infty}^{\frac{2N-\mu}{2N}}
\endaligned
\end{equation}
and
\begin{equation}\label{p153}
\aligned
\lim_{R\rightarrow\infty}&\overline{\lim}_{n\rightarrow\infty}\int_{\mathbb{R}^N}\int_{\mathbb{R}^N}\frac{|u_n(x)|
^{p}|u_n(y)|^{p}\eta_{R}(y)}
{|x-y|^{\mu}}dxdy\\
&\leq C(N,\mu)\lim_{R\rightarrow\infty}\overline{\lim}_{n\rightarrow\infty}|u_n|_{\frac{2Np}{2N-\mu}}^{p}||u_n|^{p} \eta_{R}|_{\frac{2N}{2N-\mu}}\\
&\leq C_3\lim_{R\rightarrow\infty}\overline{\lim}_{n\rightarrow\infty}\Big(\int_{\mathbb{R}^N}|u_{n}|^{2}\eta_{R}dx\Big)^{a_{p}}
\Big(\int_{\mathbb{R}^N}|u_{n}|^{2^{\ast}}\eta_{R}dx\Big)^{b_{p}}\\
&\leq C_3(\frac{V_{\infty}}{\tau_{0}})^{a_{p}}
\zeta_{\infty}^{b_{p}},
\endaligned
\end{equation}
where $C_1, C_2, C_3$ depend only on the embedding constant and the best constant $S_{H,L}$, since Lemma \ref{F0} holds and
$$
(\frac12-\frac{1}{2p})\int_{\mathbb{R}^N}(|\nabla u_n|^{2}+V(x)|u_n|^{2})dx\leq c<c_0\leq\frac{N+2-\mu}{4N-2\mu}S_{H,L}^{\frac{2N-\mu}{N+2-\mu}}+o_n(1).
$$
Now, by the definition of $F_{\infty}$, from \eqref{p151} to \eqref{p153} we know
\begin{equation}\label{p16}
F_{\infty}\leq C(\frac{V_{\infty}}{\tau_{0}})^{a_{p}}
\zeta_{\infty}^{b_{p}}+C
\zeta_{\infty}^{\frac{2N-\mu}{2N}}.
\end{equation}
Similarly, we have
\begin{equation}\label{p166}\aligned
\nu_{\infty}&=\lim_{R\rightarrow\infty}
\overline{\lim}_{n\rightarrow\infty}
\int_{\mathbb{R}^N}
\Big(\int_{\mathbb{R}^N}
\frac{|u_{n}(y)|^{2_{\mu}^{\ast}}}{|x-y|^{\mu}}dy\Big)| u_{n}(x)|^{2_{\mu}^{\ast}}\eta_{R}dx\\
&\leq C(N,\mu)\lim_{R\rightarrow\infty}\overline{\lim}_{n\rightarrow\infty}
|u_{n}|^{2^{\ast}_\mu}_{2^{\ast}}
\Big(\int_{\mathbb{R}^N}|u_{n}|^{2^{\ast}}\eta_{R}dx\Big)^{\frac{2N-\mu}{2N}}\\
&= C_4\zeta_{\infty}^{\frac{2N-\mu}{2N}}.
\endaligned
\end{equation}
Substituting \eqref{p16} and \eqref{p166} into \eqref{p13} we obtain that
\begin{equation}\label{p11}
\omega_{\infty}+V_{\infty}\leq C_5(\frac{V_{\infty}}{\tau_{0}})^{a_{p}}
\zeta_{\infty}^{b_{p}}+C_6
\zeta_{\infty}^{\frac{2N-\mu}{2N}}.
\end{equation}
Now, if $\zeta_{\infty}=0$ then it is easy to see the conclusion
$$
\nu_{\infty}=\omega_{\infty}=F_{\infty}=V_{\infty}=0.
$$
Otherwise, if $\zeta_{\infty}>0$ then applying the Young inequality to \eqref{p11} we know that there exists $\Lambda_0>0$ such that
$$
\zeta_{\infty}\geq \Lambda_0.
$$
Thus applying Lemma \ref{Concentration-compactness principle2}, we know that
$$
\omega_{\infty}\geq S\Lambda_0^{\frac{2}{2^{\ast}}}.
$$
Thus we know this is a contradiction if
$$
c< \alpha' S\Lambda_0^{\frac{2}{2^{\ast}}}.
$$

From the arguments above, let
$$
c_0=\min\{\frac{N+2-\mu}{4N-2\mu}S_{H,L}^{\frac{2N-\mu}{N+2-\mu}}, \alpha' S\Lambda_0^{\frac{2}{2^{\ast}}}\},
$$
if $c<c_0$ then we have
\begin{equation}\label{SINESC}
\nu_{\infty}=\omega_{\infty}=F_{\infty}=V_{\infty}=0, \ I=\emptyset.
\end{equation}
By using Lemma \ref{CCP1} to derive
$$
\lim_{n\rightarrow\infty}
\int_{\mathbb{R}^N}\int_{\mathbb{R}^N}\frac{|u_n(x)|
^{2_{\mu}^{\ast}}|u_n(y)|^{2_{\mu}^{\ast}}}
{|x-y|^{\mu}}dxdy
=\int_{\mathbb{R}^N}\int_{\mathbb{R}^N}\frac{|u_0(x)|
^{2_{\mu}^{\ast}}|u_0(y)|^{2_{\mu}^{\ast}}}
{|x-y|^{\mu}}dxdy,
$$
which together with Lemma \ref{BLN} imply
$$
\lim_{n\rightarrow\infty}
\int_{\mathbb{R}^N}\int_{\mathbb{R}^N}\frac{|(u_n-u_0)(x)|
^{2_{\mu}^{\ast}}|(u_n-u_0)(y)|^{2_{\mu}^{\ast}}}
{|x-y|^{\mu}}dxdy=0.
$$
\end{proof}

\subsection{Proof of Theorem \ref{EX1}}
We can verify that the functional $J$ satisfies the Mountain-Pass geometry. By Lemma \ref{F0} we have
$$\aligned
J(u)&=\frac{1}{2}\int_{\mathbb{R}^N}(|\nabla u|^{2}+V(x)|u|^{2})dx-\frac{1}{2\cdot2_{\mu}^{\ast}}\int_{\mathbb{R}^N}
\int_{\mathbb{R}^N}\frac{(|u(x)|^{2_{\mu}^{\ast}}+|u(x)|^{p})(|u(y)|^{2_{\mu}^{\ast}}+|u(y)|^{p})}
{|x-y|^{\mu}}dxdy\\
&\geq C\|u\|_{V}^{2}-C_{1}\|u\|_{V}^{2\cdot2_{\mu}^{\ast}}-C_{2}\|u\|_{V}^{2_{\mu}^{\ast}+p}
-C_{3}\|u\|_{V}^{2p}.
\endaligned$$
Since $2<2p<2_{\mu}^{\ast}+p<2\cdot2_{\mu}^{\ast}$, we can choose some $\alpha,\rho>0$ such that $J(u)\geq\alpha$ for $\|u\|_{V}=\rho$. For any $u\in E\backslash\ \{0\}$, we have
$$
J(tu)\leq\frac{t^{2}}{2}\int_{\mathbb{R}^{N}}(|\nabla u|^{2}+ V(x)u^{2}dx-\frac{t^{2\cdot2_{\mu}^{\ast}}}{2\cdot2_{\mu}^{\ast}}
\int_{\mathbb{R}^N}\int_{\mathbb{R}^N}
\frac{|u(x)|^{2_{\mu}^{\ast}}|u(y)|^{2_{\mu}^{\ast}}}
{|x-y|^{\mu}}dxdy<0
$$
for $t>0$ large enough. Hence, we can apply the mountain pass theorem without $(PS)$ condition (cf. \cite{Wi}) to get a bounded $(PS)$ sequence $\{u_{n}\}$ such that $J(u_{n})\rightarrow c^{\star}$ and $ J'(u_{n})\rightarrow0$ in $E^{-1}$ at the minimax level
$$
c^{\star}=\inf\limits_{\gamma\in\Gamma}\max\limits_{t\in[0,1]}J(\gamma(t))>0,
$$
where
$$
\Gamma:=\{\gamma\in C([0,1],E):\gamma(0)=0,J(\gamma(1))<0\}.
$$

We claim that
 \begin{equation}\label{Infimum}
\inf\,\left\{\int_{\R^N}|\nabla\varphi|^2dx: \
\varphi\in \mathcal{C}^\infty_0(\R^N),\,  \int_{\R^N}\int_{\R^N}\frac{|\varphi(x)|^{p}|\varphi(y)|^{p}}{|x-y|^{\mu}}dxdy=1\right\} = 0.\end{equation}
In fact, for all fixed $\varphi$ satisfying
$$
 \int_{\R^N}\int_{\R^N}\frac{|\varphi(x)|^{p}|\varphi(y)|^{p}}{|x-y|^{\mu}}dxdy=1,
 $$
let us define
 $$
\varphi_t=t^{\frac{2N-\mu}{2p}}\varphi(tx),\ \ t>0,
 $$
then we have
 $$
 \int_{\R^N}\int_{\R^N}\frac{|\varphi_t(x)|^{p}|\varphi_t(y)|^{p}}{|x-y|^{\mu}}dxdy
 =\int_{\R^N}\int_{\R^N}\frac{|\varphi(x)|^{p}|\varphi(y)|^{p}}{|x-y|^{\mu}}dxdy=1
 $$
 and
 $$
 \int_{\R^N}|\nabla\varphi_t|^2dx=t^{\frac{2N-\mu}{p}-N+2} \int_{\R^N}|\nabla\varphi|^2dx.
 $$
Since $\frac{2N-\mu}{p}>N-2$, we know
 $$
 \int_{\R^N}|\nabla\varphi_t|^2dx\to0
 $$
 as $t\to0$, the claim is thus proved. Now, for any $\de>0$ one can choose $\va_\de\in\mathcal{C}^\infty_0(\R^N)$ such that
$$\int_{\R^N}\int_{\R^N}\frac{|\va_\de(x)|^{p}|\va_\de(y)|^{p}}
{|x-y|^{\mu}}dxdy=1$$ and
$$|\nabla\va_\de|^2_2<\de.$$
Since $\va_\de\in\mathcal{C}^\infty_0(\R^N)$ and $V(x)\in L^{\frac{N}{2}}(\R^N)$, we have
$$
\left|\int_{\mathbb{R}^N}V(x)|\va_\de|^{2}dx\right|\leq
\left(\int_{\mathbb{R}^N}|V(x)|^{\frac{N}{2}}dx\right)^{\frac{2}{N}}
\left(\int_{\mathbb{R}^N}|\va_\de|^{2^{\ast}}dx\right)^{\frac{N-2}{N}}
\leq CS\int_{\R^N}\big|\nabla \va_\de\big|^2dx,
$$
where $S$ is the best Sobolev constant. And so,
$$
J(\va_\de)\leq\frac{1+CS}{2}\int_{\R^N}\big|\nabla \va_\de\big|^2dx-\frac{1 }{2\cdot2_{\mu}^{\ast}}\int_{\R^N}\int_{\R^N}\frac{|\va_\de(x)|^{p}|\va_\de(y)|^{p}}
{|x-y|^{\mu}}dxdy.
$$
Denote
$$
\Psi(\va_\de):=\frac{1+CS}{2}\int_{\R^N}\big|\nabla \va_\de\big|^2dx-\frac{1 }{2\cdot2_{\mu}^{\ast}}\int_{\R^N}\int_{\R^N}\frac{|\va_\de(x)|^{p}|\va_\de(y)|^{p}}
{|x-y|^{\mu}}dxdy.
$$
Then, we know
\begin{equation}\label{ES1}
\aligned
 \max_{t\in \R^+}\Psi(t\va_\de)
 &=\max_{t\in \R^+}\left\{t^{2}\frac{1+CS}{2}\int_{\R^N}\big|\nabla \va_\de\big|^2dx-\frac{t^{2p} }{2\cdot2_{\mu}^{\ast}}\int_{\R^N}\int_{\R^N}\frac{|\va_\de(x)|^{p}|\va_\de(y)|^{p}}
{|x-y|^{\mu}}dxdy\right\}\\
&=\Big(\frac{2\cdot2_{\mu}^{\ast}}{p}\Big)^{\frac{1}{p-1}}\Big(1-\frac{1}{p}\Big)
\Big(\frac{1+CS}{2}\int_{\R^N}|\nabla\va_\de|^2dx\Big)^{\frac{p}{p-1}}\\
&<\Big(\frac{2\cdot2_{\mu}^{\ast}}{p}\Big)^{\frac{1}{p-1}}\Big(1-\frac{1}{p}\Big)
\Big(\frac{1+CS}{2}\Big)^{\frac{p}{p-1}}\delta^{\frac{p}{p-1}}.
\endaligned
\end{equation}
Thus, for $c_{0}>0$ be the number given in Proposition \ref{F1}, there exists $\delta_{0}>0$ such that, for any $0<\delta<\delta_{0}$
$$
\max_{t\in \R^+}\Psi(t\va_\de)<c_{0}
$$
i.e.
$$
 \max_{t\in \R^+}J(t\va_\de)<c_{0},
$$
and so
$$
c^{\star}<c_0.
$$

Assume that $\{u_{n}\}$ converges weakly and a.e. to some weak solution $u_{0}\in E$ of \eqref{CCe1}. In particular,
\begin{equation}\label{weak}
\int_{\mathbb{R}^N}(|\nabla u_{0}|^{2}+V(x)u_{0}^{2})dx
=\int_{\mathbb{R}^N}
\int_{\mathbb{R}^N}\frac{(|u_{0}(x)|^{2_{\mu}^{\ast}}+|u_{0}(x)|^{p})
(|u_{0}(y)|^{2_{\mu}^{\ast}}
+\frac{p}{2_{\mu}^{\ast}}|u_{0}(y)|^{p})}
{|x-y|^{\mu}}dxdy.
\end{equation}
Since $
c^{\star}<c_0
$, by Proposition \ref{F1}, we have
$$
\lim_{n\rightarrow\infty}\int_{\mathbb{R}^N}\int_{\mathbb{R}^N}\frac{|(u_{n}-u_{0})(x)|^{2_{\mu}^{\ast}}
|(u_{n}-u_{0})(y)|^{2_{\mu}^{\ast}}}
{|x-y|^{\mu}}dxdy=0
$$
and
$$
\nu_{\infty}=\omega_{\infty}=F_{\infty}=V_{\infty}=0, \ I=\emptyset.
$$
So, we have
$$
\lim_{n\rightarrow\infty}
\int_{\mathbb{R}^N}V|u_{n}-u_{0}|^{2}dx=0,
$$
and
$$
\lim_{n\rightarrow\infty}
\Big(\Big(\frac{p}{2_{\mu}^{\ast}}+1\Big)\int_{\mathbb{R}^N}\int_{\mathbb{R}^N}\frac{|u_n(x)|
^{2_{\mu}^{\ast}}|u_n(y)|^{p}}
{|x-y|^{\mu}}dxdy
+\frac{p}{2_{\mu}^{\ast}}\int_{\mathbb{R}^N}\int_{\mathbb{R}^N}\frac{|u_n(x)|
^{p}|u_n(y)|^{p}}
{|x-y|^{\mu}}dxdy\Big)
$$
$$
=
\Big(\frac{p}{2_{\mu}^{\ast}}+1\Big)\int_{\mathbb{R}^N}\int_{\mathbb{R}^N}\frac{|u_0(x)|
^{2_{\mu}^{\ast}}|u_0(y)|^{p}}
{|x-y|^{\mu}}dxdy
+\frac{p}{2_{\mu}^{\ast}}\int_{\mathbb{R}^N}\int_{\mathbb{R}^N}\frac{|u_0(x)|
^{p}|u_0(y)|^{p}}
{|x-y|^{\mu}}dxdy.
$$
It follows that
$$
\aligned
0&=\lim_{n\rightarrow\infty}\langle J'(u_n),u_n\rangle\\
&=\lim_{n\rightarrow\infty}\Big(\int_{\mathbb{R}^N}(|\nabla u_n|^{2}+Vu_n^{2})dx
-\int_{\mathbb{R}^N}
\int_{\mathbb{R}^N}\frac{(|u_n(x)|^{2_{\mu}^{\ast}}+|u_n(x)|^{p})(|u_n(y)|^{2_{\mu}^{\ast}}
+\frac{p}{2_{\mu}^{\ast}}|u_n(y)|^{p})}
{|x-y|^{\mu}}dxdy\Big)\\
&=\lim_{n\rightarrow\infty}\int_{\mathbb{R}^N}|\nabla u_n|^{2}dx+\int_{\mathbb{R}^N}Vu_0^{2}dx
-\int_{\mathbb{R}^N}
\int_{\mathbb{R}^N}\frac{(|u_0(x)|^{2_{\mu}^{\ast}}+|u_0(x)|^{p})(|u_0(y)|^{2_{\mu}^{\ast}}
+\frac{p}{2_{\mu}^{\ast}}|u_0(y)|^{p})}
{|x-y|^{\mu}}dxdy.
\endaligned
$$
Combining this with \eqref{weak}, we have
$$
\lim_{n\rightarrow\infty}\int_{\mathbb{R}^N}|\nabla u_n|^{2}dx=\int_{\mathbb{R}^N}|\nabla u_0|^{2}dx.
$$
Thus,
$$
J(u_0)=\lim_{n\rightarrow\infty} J(u_n)=c^{\star}\geq\alpha>0
$$
which leads to the conclusion $u_0\neq0$.
$\hfill{} \Box$

\section{High energy solution}
In this section we assume that conditions $(V_1)$, $(V_2)$ and $(V_3)$ hold, $0<\mu<\min\{4,N\}$ and $N\geq3$.
We introduce the energy functional associated to equation \eqref{CE2} by
$$
J_{V}(u)=\frac{1}{2}\int_{\mathbb{R}^N}(|\nabla u|^{2}+ V(x)|u|^{2})dx-\frac{1}{2\cdot2_{\mu}^{\ast}}\int_{\mathbb{R}^N}\int_{\mathbb{R}^N}\frac{|u(x)|^{2_{\mu}^{\ast}}|u(y)|^{2_{\mu}^{\ast}}}
{|x-y|^{\mu}}dxdy.
$$
The Hardy--Littlewood--Sobolev inequality implies that $J_{V}$ is well defined on $D^{1,2}(\mathbb{R}^N)$ and belongs to $\mathcal{C}^{1}$. And so $u$ is a weak solution of \eqref{CE2} if and only if $u$ is a critical point of the functional $J_{V}$. To carry out the proof, we need to consider the energy functional associated to equation \eqref{CCE1} defined by
$$
J_{0}(u)=\frac{1}{2}\int_{\mathbb{R}^N}|\nabla u|^{2}dx-\frac{1}{2\cdot2_{\mu}^{\ast}}\int_{\mathbb{R}^N}
\int_{\mathbb{R}^N}\frac{|u(x)|^{2_{\mu}^{\ast}}|u(y)|^{2_{\mu}^{\ast}}}{|x-y|^{\mu}}dxdy.
$$

\subsection{A nonlocal global compactness lemma}
Let
$u\rightarrow u_{r,x_{0}}=r^{\frac{N-2}{2}}u(rx+x_{0})$ be the rescaling, where $r\in\mathbb{R}^{+}$ and $x_{0}\in \mathbb{R}^{N}$.
The following proposition is taken from \cite{DGY} which is inspired by \cite{Sm, Wi}, we sketch the proof here for readers' convenience.
\begin{lem}\label{F4} Suppose that conditions $(V_1)$, $(V_2)$ and $(V_3)$ hold and $N\geq3$, $0<\mu<\min\{4,N\}$. Assume that $\{u_{n}\}\subset D^{1,2}(\mathbb{R}^N)$ is a $(PS)$ sequence for $J_{V}$. Then there exist a number $k\in \N$, a solution $u^{0}$ of \eqref{CE2}, solutions $u^{1},... ,u^{k}$ of \eqref{CCE1}, sequences of points $x_{n}^{1},...,x_{n}^{k}\in\mathbb{R}^N$ and radii $r_{n}^{1},...,r_{n}^{k}>0$ such that for some subsequence $n\rightarrow\infty$
$$
\displaystyle\aligned
&u_{n}^{0}\equiv u_{n}\rightharpoonup u^{0}\ \ \ \mbox{weakly in}\ \ D^{1,2}(\mathbb{R}^N),\\
&u_{n}^{j}\equiv (u_{n}^{j-1}-u^{j-1})_{r_{n}^{j},x_{n}^{j}}\rightharpoonup u^{j}\ \ \ \mbox{weakly in}\ \ D^{1,2}(\mathbb{R}^N),\ \ j=1,...,k.
\endaligned
$$
Moreover as $n\rightarrow\infty$
$$
\displaystyle\aligned
&\|u_{n}\|^{2}\rightarrow \displaystyle\Sigma_{j=0}^{k}\|u^{j}\|^{2},\\
& J_{V}(u_{n})\rightarrow J_{V}(u^{0})+\displaystyle\Sigma_{j=1}^{k}J_{0}(u^{j}).
\endaligned
$$
\end{lem}
\begin{proof}
Since $\{u_{n}\}$ is a $(PS)$ sequence for $J_{V}$, we know easily that it is bounded in $D^{1,2}(\mathbb{R}^N)$. Hence we may assume that $u_{n}\rightharpoonup u^{0}$ weakly in $D^{1,2}(\mathbb{R}^N)$ as $n\rightarrow\infty$ and
that $u^{0}$ is a weak solution of \eqref{CE2}. So if we put
$$
v_{n}^{1}(x)=(u_{n}-u^{0})(x),
$$
then $v_{n}^{1}$ is a $(PS)$ sequence for $J_{V}$ satisfying
$$
v_{n}^{1}\rightharpoonup0\ \ \ \mbox{weakly in}\ \ D^{1,2}(\mathbb{R}^N).
$$
Then, together with the Br\'{e}zis-Lieb Lemma \cite{BL1} and (2.16) in \cite{BC1} that
\begin{equation}\label{b26}
\int_{\mathbb{R}^N}V(x)|v_{n}^{1}|^{2}dx\rightarrow0,
\end{equation}
 we have
\begin{equation}\label{b27}
J_{0}(v_{n}^{1})=J_{V}(v_{n}^{1})+o(1)=J_{V}(u_{n})-J_{V}(u^{0})+o(1),
\end{equation}
\begin{equation}\label{b28}
J_{0}'(v_{n}^{1})=J_{V}'(v_{n}^{1})+o(1)=o(1).
\end{equation}

If $v_{n}^{1}\rightarrow0$ strongly in $D^{1,2}(\mathbb{R}^N)$ we are done.
Now suppose
that
$$
v_{n}^{1}\nrightarrow0\ \ \ \mbox{strongly in}\ \ D^{1,2}(\mathbb{R}^N)
$$
and there exists $\gamma\in(0,\infty)$ such that
\begin{equation}\label{b29}
J_{0}(v_{n}^{1})\geq\gamma>0
\end{equation}
for $n$ large enough.

 Claim: there exist sequences $\{r_{n}\}$ and $\{y_{n}\}$ of points in $\mathbb{R}^N$ such that
\begin{equation}\label{b4}
h_{n}=(v_{n}^{1})_{r_{n},y_{n}}\rightharpoonup h\not\equiv0 \ \ \ \mbox{weakly in}\ \ D^{1,2}(\mathbb{R}^N)
\end{equation}
as $n\rightarrow\infty$.

In fact, by \eqref{b28}, we obtain
$$
J_{0}(v_{n}^{1})=\frac{N-\mu+2}{2(2N-\mu)}\int_{\mathbb{R}^N}\int_{\mathbb{R}^N}
\frac{|v_{n}^{1}(x)|^{2_{\mu}^{\ast}}|v_{n}^{1}(y)|^{2_{\mu}^{\ast}}}
{|x-y|^{\mu}}dxdy+o(1).
$$
So, by the Hardy--Littlewood--Sobolev inequality, \eqref{b29} and the boundedness of $\{u_{n}\}$, we know that $0<a_{1}<|v_{n}^{1}|_{2^{\ast}}^{2_{\mu}^{\ast}}<A_{1}$ for some $a_{1}, A_{1}>0$.
Let us define the Levy concentration function:
$$
Q_{n}(r):=\sup_{z\in\mathbb{R}^N}\int_{B_{r}(z)}|v_{n}^{1}(x)|^{2^{\ast}}dx.
$$
Since $Q_{n}(0)=0$ and $Q_{n}(\infty)>a_{1}^{\frac{2N}{2N-\mu}}$, we may assume there exists sequences $\{r_{n}\}$ and $\{y_{n}\}$ of points in $\mathbb{R}^N$ such that $r_{n}>0$ and
$$
\sup_{z\in\mathbb{R}^N}\int_{B_{r_{n}}(z)}|v_{n}^{1}(x)|^{2^{\ast}}dx
=\int_{B_{r_{n}}(y_{n})}|v_{n}^{1}(x)|^{2^{\ast}}dx=b
$$
for some
$$
0<b<\min\left\{\frac{S^{\frac{2N}{4-\mu}}}
{(2C(N,\mu)A_{1})^{\frac{2N}{4-\mu}}},a_{1}^{\frac{2N}{2N-\mu}}\right\}.
$$

Let us define $h_{n}:=(v_{n}^{1})_{r_{n},y_{n}}$. We may assume that $h_{n}\rightharpoonup h$ weakly in $D^{1,2}(\mathbb{R}^N)$ and $h_{n}\rightarrow h$ a.e. on $\mathbb{R}^N$. It is easy to see that
$$
\sup_{z\in\mathbb{R}^N}\int_{B_{1}(z)}|h_{n}(x)|^{2^{\ast}}dx
=\int_{B_{1}(0)}|h_{n}(x)|^{2^{\ast}}dx=b.
$$
By invariance of the $D^{1,2}(\mathbb{R}^N)$ norms under translation and dilation, we get
$$
\|v_{n}^{1}\|=\|h_{n}\|, \ \ |v_{n}^{1}|_{2^{\ast}}=|h_{n}|_{2^{\ast}}
$$
and
$$
\int_{\mathbb{R}^N}\int_{\mathbb{R}^N}\frac{|v_{n}^{1}(x)|^{2_{\mu}^{\ast}}|v_{n}^{1}(y)|^{2_{\mu}^{\ast}}}
{|x-y|^{\mu}}dxdy=\int_{\mathbb{R}^N}\int_{\mathbb{R}^N}\frac{|h_{n}(x)|^{2_{\mu}^{\ast}}
|h_{n}(y)|^{2_{\mu}^{\ast}}}
{|x-y|^{\mu}}dxdy.
$$
By direct calculation, we have
\begin{equation}\label{b31}
J_{0}(h_{n})=J_{0}(v_{n}^{1})=J_{V}(u_{n})-J_{V}(u^{0})+o(1)
\end{equation}
and
\begin{equation}\label{b32}
J_{0}'(h_{n})=J_{0}'(v_{n}^{1})=o(1).
\end{equation}
If $h=0$ then $h_{n}\rightarrow0$ strongly in $L_{loc}^{2}(\mathbb{R}^N)$. Let $\psi\in \mathcal{C}_{0}^{\infty}(\mathbb{R}^N)$ be such that $Supp \psi\subset B_{1}(y)$ for some $y\in\mathbb{R}^N$. Then, we have
$$
\aligned
\int_{\mathbb{R}^N}|\nabla(\psi h_{n})|^{2}dx&=\int_{\mathbb{R}^N}\nabla h_{n}\nabla(\psi^{2} h_{n})dx+o(1)\\
&=\int_{\mathbb{R}^N}\int_{\mathbb{R}^N}
\frac{|h_{n}(x)|^{2_{\mu}^{\ast}}|\psi(y)|^{2}|h_{n}(y)|^{2_{\mu}^{\ast}}}
{|x-y|^{\mu}}dxdy+o(1)\\
&\leq C(N,\mu)
|h_{n}|_{2^{\ast}}^{2_{\mu}^{\ast}}\Big(\int_{\mathbb{R}^N}
(|\psi|^{2}|h_{n}|^{2_{\mu}^{\ast}})^{\frac{2N}{2N-\mu}}dx\Big)^{\frac{2N-\mu}{2N}}+o(1)\\
&= C(N,\mu)
|h_{n}|_{2^{\ast}}^{2_{\mu}^{\ast}}\Big(\int_{\mathbb{R}^N}
|\psi h_{n}|^{\frac{4N}{2N-\mu}}
|h_{n}|^{\frac{2N(4-\mu)}{(2N-\mu)(N-2)}}dx\Big)^{\frac{2N-\mu}{2N}}+o(1)\\
&\leq C(N,\mu)|h_{n}|_{2^{\ast}}^{2_{\mu}^{\ast}}
|h_{n}|_{L^{2^{\ast}}(B_{1}(y))}^{2_{\mu}^{\ast}-2}
\frac{1}{S}\int_{\mathbb{R}^N}|\nabla(\psi h_{n})|^{2}dx+o(1)\\
&\leq C(N,\mu)b^{\frac{2_{\mu}^{\ast}-2}{2^{\ast}}}
\frac{A_{1}}{S}\int_{\mathbb{R}^N}|\nabla(\psi h_{n})|^{2}dx+o(1)\\
&\leq\frac{1}{2}\int_{\mathbb{R}^N}|\nabla(\psi h_{n})|^{2}dx+o(1)
\endaligned
$$
thanks to $0<\mu<\min\{4,N\}$. We obtain $\nabla h_{n}\rightarrow0$ strongly in $L_{loc}^{2}(\mathbb{R}^N)$ and $h_{n}\rightarrow0$ strongly in $L_{loc}^{2^{\ast}}(\mathbb{R}^N)$, which contradicts with $\displaystyle\int_{B_{1}(0)}|h_{n}(x)|^{2^{\ast}}dx=b>0$. So, $h\neq0$. By \eqref{b28} and weakly sequentially continuous $J_{0}'$, we know $h$ solves \eqref{CCE1} weakly. The sequences $\{h_{n}\}$, $\{r_{n}^{1}\}$, and $\{y_{n}^{1}\}$ are the wanted sequences.

By iteration, we obtain sequences $v_{n}^{j}=u_{n}^{j-1}-u^{j-1}$, $j\geq2$, and the rescaled functions $u_{n}^{j}=(v_{n}^{j})_{r_{n}^{j},y_{n}^{j}}\rightharpoonup u^{j}$ weakly in $ D^{1,2}(\mathbb{R}^N)$, where
each $u^{j}$ solves \eqref{CCE1}. By induction we know that
\begin{equation}\label{b5}
\|u_{n}^{j}\|^{2}=
\|u_{n}\|^{2}-\Sigma_{i=0}^{j-1}\|u^{i}\|^{2}+o(1)
\end{equation}
and
\begin{equation}\label{b6}
J_{0}(u_{n}^{j})
=J_{V}(u_{n})-J_{V}(u^{0})-\Sigma_{i=1}^{j-1}J_{0}(u^{i})+o(1).
\end{equation}
Furthermore, from the estimate
$$
0=\langle J_{0}'(u^{j}),u^{j}\rangle=\|u^{j}\|^{2}-\int_{\mathbb{R}^N}
\int_{\mathbb{R}^N}\frac{|u^{j}(x)|^{2_{\mu}^{\ast}}|u^{j}(y)|^{2_{\mu}^{\ast}}}{|x-y|^{\mu}}dxdy
\geq\|u^{j}\|^{2}(1-S_{H,L}^{-2_{\mu}^{\ast}}\|u^{j}\|^{2\cdot2_{\mu}^{\ast}-2}),
$$
we see that $\|u^{j}\|^{2}\geq S_{H,L}^{\frac{2N-\mu}{N-\mu+2}}$ and the iteration must terminate at some index $k\geq0$ due to \eqref{b5}.
\end{proof}

Let
$$
P(u)=\int_{\mathbb{R}^N}(|\nabla u|^{2}+ V(x)|u|^{2})dx
$$
and
$$
\mathcal{M}=\{u\in D^{1,2}(\mathbb{R}^N):\|u\|_{NL}=1\}.
$$
\begin{Prop}\label{NE}
Suppose that conditions $(V_1)$, $(V_2)$ and $(V_3)$ hold. Then the minimization problem
\begin{equation}\label{a3}
\inf\{P(u):u\in\mathcal{M}\}
\end{equation}
has no solution.

\end{Prop}
\begin{proof}
Let denote by $S_{\mathcal{M}}$ the infimum defined by \eqref{a3}. Obviously $S_{\mathcal{M}}\geq S_{H,L}$. First we shall show that actually the equality holds. Let us consider
the sequence
$$
\varphi_{\frac{1}{n},0}(x)=S_{H,L}^{\frac{2-N}{2(N-\mu+2)}}U_{\frac{1}{n},0}(x)
=S_{H,L}^{\frac{2-N}{4}}C(N,\mu)^{\frac{N(2-N)}{4(2N-\mu)}}
\frac{[N(N-2)\frac{1}{n}]^{\frac{N-2}{4}}}{(\frac{1}{n}+|x|^{2})^{\frac{N-2}{2}}}
$$
then $\forall p\in(\frac{N}{N-2},\frac{2N}{N-2})$, $|\varphi_{\frac{1}{n},0}(x)|_{p}\rightarrow0$ (see (2.4), \cite{BC1}), in fact
$$\aligned
|\varphi_{\frac{1}{n},0}(x)|_{p}^{p}&=\left[\frac{N(N-2)}{S_{H,L}}\right]^{\frac{(N-2)p}{4}}
C(N,\mu)^{\frac{N(2-N)p}{4(2N-\mu)}}
\int_{\mathbb{R}^{N}}\frac{(\frac{1}{n})^{\frac{(N-2)p}{4}}}{(\frac{1}{n}+|x|^{2})^{\frac{(N-2)p}{2}}}dx\\
&=\left[\frac{N(N-2)}{S_{H,L}}\right]^{\frac{(N-2)p}{4}}
C(N,\mu)^{\frac{N(2-N)p}{4(2N-\mu)}}\left(\frac{1}{n}\right)^{\frac{N}{2}-\frac{(N-2)p}{4}}
\int_{\mathbb{R}^{N}}\frac{1}{(1+|x|^{2})^{\frac{(N-2)p}{2}}}dx.
\endaligned$$
Moreover using the definition of $S_{H,L}$ and the fact that $U_{\frac{1}{n},0}$ solves \eqref{CCE1} it is easy to verify that
$$
\varphi_{\frac{1}{n},0}\in\mathcal{M},\ i.e.\ \|\varphi_{\frac{1}{n},0}\|_{NL}=1.
$$
Now using the H\"{o}lder inequality with $p\in(\frac{N}{2},p_{2})$ we get
$$
P(\varphi_{\frac{1}{n},0})=\int_{\mathbb{R}^{N}}|\nabla\varphi_{\frac{1}{n},0}|^{2}dx+
\int_{\mathbb{R}^{N}}V(x)|\varphi_{\frac{1}{n},0}|^{2}dx\leq S_{H,L}+|V(x)|_{p}|\varphi_{\frac{1}{n},0}(x)|_{2p'}^{2}.
$$
Since $2p'\in(\frac{N}{N-2},\frac{2N}{N-2})$, we can obtain $S_{\mathcal{M}}=S_{H,L}$.

Now it is easy to prove the nonexistence result arguing by contradiction.
Let $u\in\mathcal{M}$ be a function such that
$$
P(u)=S_{H,L}.
$$
If $\displaystyle\int_{\mathbb{R}^{N}}V(x)|u|^{2}dx>0$, then we have
$$
\int_{\mathbb{R}^{N}}|\nabla u|^{2}dx<\int_{\mathbb{R}^{N}}|\nabla u|^{2}dx+
\int_{\mathbb{R}^{N}}V(x)|u|^{2}dx=S_{H,L}
$$
contradicting the definition of $S_{H,L}$. If $\displaystyle\int_{\mathbb{R}^{N}}V(x)|u|^{2}dx=0$, then
$$
\int_{\mathbb{R}^{N}}|\nabla u|^{2}dx=S_{H,L}
$$
and $S_{H,L}^{\frac{N-2}{2(N-\mu+2)}}u$ is a solution of \eqref{CCE1}. Recall that any solution of \eqref{CCE1} must be of the form
$$
U_{\delta,z}(x)=C(N,\mu)^{\frac{2-N}{2(N-\mu+2)}}S^{\frac{(N-\mu)(2-N)}{4(N-\mu+2)}}
\frac{[N(N-2)\delta]^{\frac{N-2}{4}}}{(\delta+|x-z|^{2})^{\frac{N-2}{2}}}, \ \delta>0, \ z\in\mathbb{R}^{N},
$$
then we know
$$
u=C(N,\mu)^{\frac{2-N}{2(2N-\mu)}}S^{\frac{2-N}{4}}
\frac{[N(N-2)\delta^{1}]^{\frac{N-2}{4}}}{(\delta^{1}+|x-z^{1}|^{2})^{\frac{N-2}{2}}}
$$
for some $\delta^{1}>0$ and $z^{1}\in\mathbb{R}^{N}$. Since $V(x)\geq0$ on $\mathbb{R}^{N}$ and $V(x) > 0$ in a positive measure set, we have
$$
\int_{\mathbb{R}^{N}}V(x)|u|^{2}dx>0,
$$
which contradicts with $\displaystyle\int_{\mathbb{R}^{N}}V(x)|u|^{2}dx=0$.

So in conclusion, we know that $S_{\mathcal{M}}$ is not attained.
\end{proof}

\begin{cor}\label{PS3}
The functional $P|_{\mathcal{M}}$ satisfies the $(PS)_{c}$-condition for $c\in(S_{H,L},2^{\frac{N+2-\mu}{2N-\mu}}S_{H,L})$.
\end{cor}
\begin{proof}
Let $\{u_{n}\}\subset D^{1,2}(\mathbb{R}^N)$ be a $(PS)_{c}$-sequence for $P|_{\mathcal{M}}$ with $c\in(S_{H,L},2^{\frac{N+2-\mu}{2N-\mu}}S_{H,L})$.
Then, $\{w_{n}\}$ is a $(PS)_{c}$-sequence for $J_{V}$ with
$$\frac{N+2-\mu}{4N-2\mu} S_{H,L}^{\frac{2N-\mu}{N+2-\mu}}<c<
\frac{N+2-\mu}{2N-\mu}S_{H,L}^{\frac{2N-\mu}{N+2-\mu}},
$$
where $w_{n}=P(u_{n})^{\frac{N-2}{2(N+2-\mu)}}u_{n}$. We know from Lemma \ref{F4} that there exist a number $k\in \N$, a solution $w^{0}$ of \eqref{CE2} and solutions $w^{1},... ,w^{k}$ of \eqref{CCE1}, such that for some subsequence $n\rightarrow\infty$
$$
\aligned
&\|w_{n}\|^{2}\rightarrow \Sigma_{j=0}^{k}\|w^{j}\|^{2},\\
&J_{V}(w_{n})\rightarrow J_{V}(w^{0})+\Sigma_{j=1}^{k}J_{0}(w^{j}).
\endaligned
$$
By Proposition \ref{NE}, if $w$ is a nontrivial solution of \eqref{CE2}, then
$$
J_{V}(w)>\frac{N+2-\mu}{2(2N-\mu)} S_{H,L}^{\frac{2N-\mu}{N+2-\mu}}.
$$
While for every nontrivial solution $v$ of \eqref{CCE1}
$$
J_{0}(v)\geq\frac{N+2-\mu}{2(2N-\mu)} S_{H,L}^{\frac{2N-\mu}{N+2-\mu}}.
$$
Since
$$c<
\frac{N+2-\mu}{2N-\mu}S_{H,L}^{\frac{2N-\mu}{N+2-\mu}},
$$ we have $k=0$ or $k=1$ with $w^{0}=0$.
In conclusion, $\{w_{n}\}$ is relatively compact in $D^{1,2}(\mathbb{R}^N)$.

So, the functional $P|_{\mathcal{M}}$ satisfies the $(PS)_{c}$-condition for $c\in(S_{H,L},2^{\frac{N+2-\mu}{2N-\mu}}S_{H,L})$.
\end{proof}

\subsection{Proof of Theorem \ref{EXS}}
We now consider the functions
$$
\varphi_{\delta,z}(x)=\frac{U_{\delta,z}(x)}{\|U_{\delta,z}(x)\|_{NL}}=
S_{H,L}^{\frac{2-N}{4}}C(N,\mu)^{\frac{N(2-N)}{4(2N-\mu)}}
\frac{[N(N-2)\delta]^{\frac{N-2}{4}}}{(\delta+|x-z|^{2})^{\frac{N-2}{2}}},\ \ \delta>0,\ \ z\in\mathbb{R}^{N}.
$$
Note that $\forall \delta>0,z\in\mathbb{R}^{N}$
$$
\|\varphi_{\delta,z}\|^{2}=S_{H,L},\ \ \  \|\varphi_{\delta,z}\|_{NL}=1
$$
and so $\varphi_{\delta,z}\in\mathcal{M}$. Moreover $|\varphi_{\delta,z}|_{p}$, $ p\in(\frac{N}{N-2},\frac{2N}{N-2})$, for any fixed $p$ depends only on $\delta$ because of the invariance by translation of the $L^{p}(\mathbb{R}^{N})$ norm.

\begin{lem}\label{Pb2}
Suppose that $V(x)$ satisfies $(V_3)$. Then
$$
P(\varphi_{\delta,z})<2^{\frac{N+2-\mu}{2N-\mu}}S_{H,L},\ \ \forall\delta>0,\ \ \forall z\in\mathbb{R}^{N}.
$$
\end{lem}
\begin{proof}
Using $(V_3)$, the H\"{o}lder inequality and the Hardy--Littlewood--Sobolev inequality, we get
$$
\aligned
P(\varphi_{\delta,z})&=\int_{\mathbb{R}^{N}}|\nabla\varphi_{\delta,z}|^{2}dx+
\int_{\mathbb{R}^{N}}V(x)|\varphi_{\delta,z}|^{2}dx\\
&\leq S_{H,L}+|V(x)|_{\frac{N}{2}}\Big(\int_{\mathbb{R}^{N}}|\varphi_{\delta,z}|^{\frac{2N}{N-2}}dx\Big)^{\frac{N-2}{N}}\\
&= S_{H,L}+|V(x)|_{\frac{N}{2}}\frac{1}{C(N,\mu)^{\frac{N-2}{2N-\mu}}}
\Big(\int_{\mathbb{R}^N}\int_{\mathbb{R}^N}\frac{|\varphi_{\delta,z}(x)|^{2_{\mu}^{\ast}}
|\varphi_{\delta,z}(y)|^{2_{\mu}^{\ast}}}{|x-y|^{\mu}}dxdy\Big)^{\frac{N-2}{2N-\mu}}\\
&<S_{H,L}+(2^{\frac{N+2-\mu}{2N-\mu}}-1)S_{H,L}=2^{\frac{N+2-\mu}{2N-\mu}}S_{H,L}.
\endaligned
$$
\end{proof}

Now put
$$
\phi(x)=\left\{\begin{array}{l}
\displaystyle 0\hspace{4.14mm}\mbox{if}\hspace{1.14mm} |x|<1,\\
\displaystyle 1\hspace{4.14mm}\mbox{if}\hspace{1.14mm} |x|\geq1,
\end{array}
\right.
$$
and define
$$
\alpha:D^{1,2}(\mathbb{R}^N)\rightarrow\mathbb{R}^{N+1}
$$
$$
\alpha(u)=\frac{1}{S_{H,L}}\int_{\mathbb{R}^{N}}\Big(\frac{x}{|x|},\phi(x)\Big)|\nabla u|^{2}dx=(\beta(u),\gamma(u)),
$$
where
$$
\beta(u)=\frac{1}{S_{H,L}}\int_{\mathbb{R}^{N}}\frac{x}{|x|}|\nabla u|^{2}dx
$$
and
$$
\gamma(u)=\frac{1}{S_{H,L}}\int_{\mathbb{R}^{N}}\phi(x)|\nabla u|^{2}dx.
$$

Denote
$$
\mathcal{A}:=\{u\in\mathcal{M}:\alpha(u)=(0,\frac{1}{2})\},
$$
and
$$
c^{\star}=\inf_{u\in\mathcal{A}}P(u).
$$

The following proposition is due to Benci and Cerami \cite{BC1} with $S$ replaced by $S_{H,L}$.

\begin{Prop}\label{Pb3}
\begin{itemize}
\item[(1).] $c^{\star}>S_{H,L}$;
\item[(2).] There is a $\delta_{1}:0<\delta_{1}<\frac{1}{2}$ such that
$$\aligned
P(\varphi_{\delta_{1},z})&<\frac{S_{H,L}+c^{\star}}{2},\hspace{4.14mm}\forall z\in\mathbb{R}^{N},\\
\gamma(\varphi_{\delta_{1},z})&<\frac{1}{2},\hspace{4.14mm}\forall z:|z|<\frac{1}{2},\\
\left|\beta(\varphi_{\delta_{1},z})-\frac{z}{|z|}\right|&<\frac{1}{4},\hspace{4.14mm}\forall z:|z|\geq\frac{1}{2};
\endaligned$$
\item[(3).] There is a $\delta_{2}:\delta_{2}>\frac{1}{2}$ such that
$$\aligned
P(\varphi_{\delta_{2},z})&<\frac{S_{H,L}+c^{\star}}{2},\hspace{4.14mm}\forall z\in\mathbb{R}^{N},\\
\gamma(\varphi_{\delta_{2},z})&>\frac{1}{2},\hspace{4.14mm}\forall z\in\mathbb{R}^{N};
\endaligned$$
\item[(4).] There exists $R\in\mathbb{R}^{+}$ such that
$$\aligned
P(\varphi_{\delta,z})&<\frac{S_{H,L}+c^{\star}}{2},\hspace{4.14mm}\forall z:|z|\geq R\ \ \mbox{and}\ \ \delta\in[\delta_{1},\delta_{2}],\\
(\beta(\varphi_{\delta,z})|z)_{\mathbb{R}^{N}}&>0,\hspace{4.14mm}\forall z:|z|\geq R\ \ \mbox{and}\ \ \delta\in[\delta_{1},\delta_{2}].
\endaligned$$
\end{itemize}
\end{Prop}

Now let
$$
Z=\{(z,\delta)\in \mathbb{R}^{N+1}:|z|<R,\delta\in[\delta_{1},\delta_{2}]\},
$$
and let $\Phi$ be the operator
$$
\Phi:[\mathbb{R}^{N}\times(0,+\infty)]\rightarrow D^{1,2}(\mathbb{R}^{N})
$$
given by
$$
\Phi(z,\delta)=\varphi_{\delta,z}(x).
$$
Note that $\Phi$ is continuous. Call $\Sigma$ the subset of $\mathcal{M}$ defined by
$$
\Sigma=\{\Phi(z,\delta):(z,\delta)\in \overline{Z}\}.
$$
Consider then the family
$$
\mathbb{A}:=\{h\in L(\mathcal{M},\mathcal{M}) : h(u)=u,\forall u\in P^{-1}((-\infty,\frac{S_{H,L}+c^{\star}}{2}))\}
$$
and define
$$
\Gamma=\{B\subset\mathcal{M}:B=h(\Sigma),h\subset\mathbb{A}\}.
$$

Similar to the proof of Lemma 3.12 \cite{BC1}, we know that
\begin{lem}\label{Pb10}
If $B\in\Gamma$, then $B\cap\mathcal{A}\neq\emptyset$.
\end{lem}

Now we set
\begin{equation}\label{E7}
c=\inf_{B\in\Gamma}\sup_{u\in B}P(u)
\end{equation}
$$
K_{c}=\{u\in \mathcal{M}:P(u)=c\mbox{ and } P'|_{\mathcal{M}}(u)=0\}.
$$
Moreover for $d\in\mathbb{R}$, $P^{d}$ will be
$$
P^{h}=\{u\in \mathcal{M}:P(u)\leq d\}.
$$

\noindent
{\bf Proof of Theorem \ref{EXS}.}
We shall prove the theorem showing that $K_{c}\neq\emptyset$, i.e., that $c$ defined by \eqref{E7} is a critical level and there is a critical point $u$ such that $P(u)=c$. By $\Sigma\in \Gamma$ and Lemma \ref{Pb2}, we know
$$
c\leq\sup_{u\in \Sigma}P(u)\leq\sup_{z\in\mathbb{R}^{N},\delta\in\mathbb{R}^{+}}
P(\varphi_{\delta,z})<2^{\frac{N+2-\mu}{2N-\mu}}S_{H,L}.
$$
Also by Lemma \ref{Pb10}, $B\cap\mathcal{A}\neq\emptyset$, $\forall B\in\Gamma$, so
$$
c\geq\inf_{\mathcal{A}}P(u)=c^{\star}>S_{H,L}.
$$
Hence
$$
S_{H,L}<c^{\star}<2^{\frac{N+2-\mu}{2N-\mu}}S_{H,L}.
$$
Suppose now $K_{c}=\emptyset$. By Proposition \ref{PS3} the Palais-Smale condition holds
in
$$\{u\in\mathcal{M}:S_{H,L}<P(u)<2^{\frac{N+2-\mu}{2N-\mu}}S_{H,L}\},$$
then using a variant of a well-known deformation Lemma (see \cite{Wi}) we
find a continuous map
$$
\eta:[0,1]\times\mathcal{M}\rightarrow\mathcal{M}
$$
and a positive number $\varepsilon_{0}$ such that
$$
P^{c+\varepsilon_{0}}\backslash P^{c-\varepsilon_{0}}\subset P^{2^{\frac{N+2-\mu}{2N-\mu}}S_{H,L}}\backslash P^{\frac{S_{H,L}+c^{\star}}{2}},
$$
$$
\eta(0,u)=u,
$$
$$
\eta(t,u)=u,\ \ \forall u\in P^{c-\varepsilon_{0}}\cup\{\mathcal{M}\backslash P^{c+\varepsilon_{0}}\},\forall t\in(0,1)
$$
and
$$
\eta(1,P^{c+\varepsilon_{0}/2})\subset P^{c-\varepsilon_{0}/2}.
$$
Now let $\widetilde{B}\in\Gamma$ be such that
$$
c\leq\sup_{\widetilde{B}}P(u)<c+\frac{\varepsilon_{0}}{2}.
$$
Then $\eta(1,\widetilde{B})\in\Gamma$ and
$$
\sup_{u\in\eta(1,\widetilde{B})}P(u)<c-\frac{\varepsilon_{0}}{2}
$$
contradicting with the definition of $c$. $\hfill{} \Box$

\end{document}